\documentclass[11pt,reqno]{amsart}
    \usepackage{amssymb}
    \usepackage{amsmath}
    \usepackage{mathrsfs}
    \usepackage{amssymb, amsmath, amsfonts}
    \usepackage[title]{appendix}
    \usepackage{mathabx}
   \usepackage{color}

    \usepackage[hidelinks]{hyperref}
    \allowdisplaybreaks
    \usepackage[numbers,sort&compress]{natbib}

    \makeatletter
    \def\tank#1{\mathbb Protected@xdef\@thanks{\@thanks
     \mathbb Protect\footnotetext[0]{#1}}}
    \def\bigfoot{

     \@footnotetext}
    \makeatother

    \newcommand{\ea}{\end{array}}

    \allowdisplaybreaks
    \numberwithin{equation}{section}

    \allowdisplaybreaks

    \newtheorem{theorem}{Theorem}[section]
    \newtheorem{lemma}{Lemma}[section]
    \newtheorem{proposition}[theorem]{Proposition}

    \newtheorem{corollary}[theorem]{Corollary}
    \newtheorem{definition}[theorem]{Definition}

    \def\beq{\begin{equation}}
    \def\nneq{\end{equation}}

    \def\bthm{\begin{theorem}}
    \def\nthm{\end{theorem}}

    \def\blem{\begin{lemma}}
    \def\nlem{\end{lemma}}
    \def\bprf{\begin{proof}}
    \def\nprf{\end{proof}}
    \def\bprop{\begin{prop}}
    \def\nprop{\end{prop}}
    \def\brmk{\begin{rem}}
    \def\nrmk{\end{rem}}

    \def\bexa{\begin{exa}}
    \def\nexa{\end{exa}}
    \def\bcor{\begin{cor}}
    \def\ncor{\end{cor}}

    \title[Stochastic  fractional  heat equation  with spatially rough noise]{Stochastic fractional heat equation  with general rough noise}

         \date{}
    
\begin{document}

       \author[B. Qian]{Bin Qian}
    \address[]{Bin Qian, Department of Mathematics and Statistics, Suzhou University of Technology, Changshu, Jiangsu, 215500,  China.}
    \email{binqiancn@126.com}

    \author[R. Wang]{Ran Wang}
    \address[]{Ran Wang, School of Mathematics and Statistics,  Wuhan University,  Wuhan, 430072,
    China.}
    \email{rwang@whu.edu.cn}

    \maketitle
     \noindent {\bf Abstract:} 
     Consider the following nonlinear one-dimensional stochastic  fractional  heat equation  
$$\frac{\partial }{\partial t}u(t, x)= -(-\Delta)^{\alpha/2}u(t, x) +\sigma(t,x,u(t,x)) \dot{W}(t, x), $$ 
 where $-(-\Delta)^{\alpha/2}$  is the fractional Laplacian   on $\mathbb R$ for   $1<\alpha<2$,   and   $\dot{W}$ is a Gaussian noise that is  white in time and behaves in space as a fractional Brownian motion  with Hurst index $H$ satisfying  $\frac{3-\alpha}{4}<H<\frac12$.   
 When $\alpha=2$, Hu and Wang ({\it Ann. Inst. Henri Poincar\'e Probab. Stat.} {\bf 58} (2022) 379-423) studied the well-posedness of the solution and its  H\"older continuity, removing the technical  condition $\sigma(0)=0$ that was previously assumed in Hu et al.  ({\it Ann. Probab.} {\bf 45} (2017) 4561-4616). Their approach relied on working in a weighted space with a suitable power decay function.

 For  the case $\alpha\in (1,2)$, inspired by Hu and Wang, we  investigate  the well-posedness of the  stochastic  fractional  heat equation without imposing  the technical  condition of $\sigma(0)=0$, which was required  in the earlier work of    Liu and Mao  ({\it Bull. Sci. Math.}  {\bf181}    (2022)  103207).  In our analysis,   precise estimates of the heat kernel associated with the fractional Laplacian $-(-\Delta)^{\alpha/2}$  play a crucial role.

      \vskip0.3cm
 \noindent{\bf Keywords:} {Stochastic  fractional heat equation; Weak solution; Strong solution; Heat kernel estimates;  H\"older continuity.}
 \vskip0.3cm

\noindent {\bf MSC: } {60H15; 60G15; 60G22.}
        \vskip0.5cm

\section{Introduction and main results}
 Consider the following nonlinear   stochastic fractional heat equation (SFHE, for short): 
 \begin{equation}\label{SFHE}
    \frac{\partial }{\partial t}u(t, x)= -(-\Delta)^{\alpha/2}u(t, x) +\sigma(t,x,u(t,x)) \dot{W}(t, x), \ \ \ t\ge0, \, x\in \mathbb R,
    \end{equation}
with initial condition $u(0,x)=u_0(x)$. Here,  $-(-\Delta)^{\alpha/2}$ denotes  the fractional Laplacian of order ${\alpha}/{2}\in (1/2,1)$, and $W(t, x)$ is a centered Gaussian process with  covariance 
\begin{equation}\label{CovW}
  \mathbb{E}\left[W(t, x)W(s, y)\right]=\frac{1}{2} \left(s\wedge t\right)\left( |x|^{2H}+|y|^{2H}-|x-y|^{2H} \right),
\end{equation}
for some $H$ satisfying $ \frac{3-\alpha}{4}<H<\frac{1}{2}$. That is, $W$ is a standard Brownian motion in time and a fractional Brownian motion (fBm, for short) with Hurst index $H$ in space, and $\dot W(t,x)=\frac{\partial^2}{\partial t\partial x}W(t,x)$ is its formal derivative.      Formally,  the covariance of the noise $\dot{W}$ is given by
    $$\mathbb E\left[\dot{W}(t,x)\dot{W}(s,y)\right]=\delta_0 ( t-s)\Lambda\left(x-y\right),$$
    where the spatial covariance $\Lambda$ is a distribution, whose Fourier transform is the measure 
    $$\mu(d\xi)=c_{H}|\xi|^{1-2H}d\xi,$$ with
         \begin{align} \label{e.c1}
   c_{H}:=\frac{1}{2\pi}\Gamma(2H+1)\sin(\pi H).
 \end{align}
     
   The spatial covariance $\Lambda(x-y)$ can be formally written as
    $$\Lambda(x-y)=H(2H-1) |x-y|^{2H-2}.$$ 
   However,   $\Lambda$ is not locally integrable and fails to be nonnegative when $H\in(\frac 14,\frac 12)$.  It does not satisfy the classical Dalang  condition in \cite{Dalang1999}, where $\Lambda$ is given by a nonnegative   locally integrable function.  Consequently, the standard approaches used in references \cite{DaPrato2014, Dalang1999, DQ, Walsh1986} do not apply to such rough covariance structures.

         Recently, many authors have studied the existence and uniqueness of solutions of stochastic partial differential equations driven by Gaussian noise with the covariance of a fractional Brownian motion with Hurst parameter $H\in (\frac1 4, \frac12)$ in the space variable. See, e.g., \cite{HHLNT2017, HHLNT2018, HW2022, LHW2022, SongSX2020} and references therein. For surveys on the subject, we refer to \cite{Hu19} and \cite{Song2018}.  For example, when $\alpha=2$ and  the diffusion coefficient is affine (i.e.,   $\sigma(x)=ax+b$), Balan et al.   \cite{BJQ2015} proved the existence and uniqueness of the mild solution to  equation \eqref{SFHE} using the Fourier analytic techniques.  They also established the H\"older continuity of the solution in \cite{BJQ2016}.  For  a nonlinear coefficient $\sigma(u)$,  Hu et al. \cite{HHLNT2017}  proved the well-posedness of  equation \eqref{SFHE} under the assumptions that  $\sigma(u)$ is Lipschitz continuous,  differentiable with  a Lipschitz  derivative, and that $\sigma(0)=0$.  Under similar conditions,  Liu and Mao \cite{LM2022} studied the well-posedness and intermittency of  the stochastic fractional heat equation.
   
          For the stochastic heat equation \eqref{SFHE} (i.e., $\alpha=2$), it follows from  \cite[Theorem 4.5]{HHLNT2017} that  the condition $\sigma(0)=0$ ensures  the solution belongs to the space $\mathcal{Z}_T^p$ (see \eqref{eq norm Lam}  below with $\lambda(x)\equiv 1$). However,  even in the additive noise case (i.e.,  $\sigma\equiv 1$), the solution $u_{\text{add}}$ is no longer in $\mathcal{Z}_T^p$. To determine whether $u_{\text{add}}\in\mathcal{Z}_T^\infty$, Hu and Wang \cite{HW2022} studied the sharp growth of $\sup_{|x|\le L}|u_{\text{add}}|$ as $L\to\infty$ using majorizing measures. For $\alpha\in(1,2)$, the sharp growth was established in \cite{LQW25}.

To remove the restriction $\sigma(0)=0$, Hu and Wang \cite{HW2022} introduced a decay weight to enlarge the solution space from $\mathcal{Z}_T^p$ to a weighted space $\mathcal{Z}_{\lambda,T}^p$,  consisting of all random fields $\{v(t,x)\}_{t\ge0,\,x\in\mathbb{R}}$ for which  the following norm is finite:

     \begin{equation}\label{eq norm Lam}
\|v\|_{\mathcal{Z}_{\lambda,T}^p}:=\sup_{t\in[0,T]}\left\|v(t,\cdot)\right\|_{L_{\lambda}^p(\Omega\times\mathbb{R})}+\sup_{t\in[0,T]}\mathcal{N}_{\frac12-H,p}^*v(t),
\end{equation}
where $p\ge1$,    $\lambda(x)=c_H(1+|x|^2)^{H-1}$ satisfies $\int_{\mathbb{R}}\lambda(x)dx=1$,
\begin{equation}\label{eq norm Lam0}
\left\|v(t,\cdot)\right\|_{L_{\lambda}^p(\Omega\times\mathbb{R})}:=\left(\int_{\mathbb R}\mathbb{E}\left[|v(t,x)|^p\right]\lambda(x)dx\right)^{\frac{1}{p}},
\end{equation}
and
\begin{equation}\label{eq norm Lam*}
\mathcal{N}_{\frac12-H,p}^*v(t):=\left(\int_{\mathbb R}\| v(t,\cdot+h)-v(t,\cdot)\|_{L_{\lambda}^p(\Omega\times\mathbb{R})}^2|h|^{2H-2}dh\right)^{\frac12}.
\end{equation}
    When $\lambda(x)\equiv1$, the corresponding space is denoted by $\mathcal{Z}_{T}^p$.  When the function is independent of $t$, the corresponding space is denoted by $\mathcal{Z}_{\lambda,0}^p$.

 Inspired by Hu and Wang  \cite{HW2022}, we study  the well-posedness of the stochastic fractional heat equation \eqref{SFHE} without the  restriction  of $\sigma(0)=0$, which was previously assumed in \cite{LM2022}.
 
  In our analysis,   precise  estimates of the fractional heat kernel  play a crucial role.   To this end, we generalize the sharp bounds on the  Gaussian heat kernel  obtained in \cite[Lemma 2.10 and Lemma 2.11]{HW2022} to the heat kernel  associated with the fractional Laplacian for $1<\alpha<2$; see Lemmas \ref{prop alpha esti1} and  \ref{prop alpha esti2} below.   In the case   $\alpha=2$, the proofs   in \cite{HW2022} rely on the Fourier transform  $\exp(-t|\cdot|^2)$ of the Gaussian heat kernel, where the specific value  ``$\alpha=2$" plays a crucial role. For the fractional Laplacian, however, the corresponding parameter   $\alpha\in (1, 2)$,   this approach is not directly applicable.  We therefore propose a novel method (see Section \ref{subsec Aux})   that allows us to estimate the relevant integrals directly, without passing to the Fourier domain, and this method may be applicable in more general settings.     Additionally, analogous to the treatment in \cite{QWWX2026} for the case $\alpha=2$ and $\sigma(0)=0$, Lemmas \ref{prop alpha esti1} and \ref{prop alpha esti2} can be employed to investigate the asymptotic behavior of the temporal increment $u(t+\varepsilon, x)-u(t, x)$ for fixed $t \ge 0$ and $x \in \mathbb{R}$ as $\varepsilon \downarrow 0$, and to extend the analysis to the framework of Liu and Mao \cite{LM2022} for $\alpha \in (1,2)$ and $\sigma(0)=0$.

The definitions of strong and weak solutions are given in Section \ref{sec solution}.  We make the following assumption for the existence      of a weak solution.
\begin{itemize}
\item[(H1)] $\sigma(t,x,u)$ is jointly continuous on $[0,T]\times \mathbb{R}^2$ and is at most of linear growth in $u$ uniformly in $t$ and $x$. That is, there exists a constant $C>0$ such that
\begin{equation}\label{eq sigma 1}
\sup_{t\in[0,T],\, x\in\mathbb{R}}\left|\sigma(t,x,u)\right|\le C(|u|+1),  \quad \ u\in \mathbb R. 
\end{equation}
  We also assume that $\sigma(t, x, u)$ is uniformly Lipschitzian  in $u$; that is,  there exists a constant $C>0$ such that  
 \begin{equation}\label{eq sigma 2}
\sup_{t\in[0,T],\, x\in\mathbb{R}}\left|\sigma(t,x,u)-\sigma(t,x,v)\right|\le C|u-v|, \quad u,v\in\mathbb{R}.
\end{equation}
  \end{itemize}
 
\begin{theorem}\label{thm main1}
Let $ \frac{3-\alpha}{4}<H<\frac 12$ and  $\lambda(x)=c_H(1+|x|^2)^{H-1}$ satisfy $\int_{\mathbb{R}}\lambda(x)dx=1$.  Assume that $\sigma(t,x,u)$ satisfies hypothesis (H1) and that the  initial datum $u_0$ belongs to $\mathcal{Z}_{\lambda,0}^p$ for some $p>\frac{2(\alpha+1)}{4H-3+\alpha}$.
  Then there exists a unique weak solution to \eqref{SFHE} whose sample paths lie in $\mathcal{C}([0,T]\times\mathbb{R})$ almost surely. Moreover, for any $0<\gamma<\frac{2H+\alpha-2}{2}-\frac{\alpha+1}{p}$, the process $u(\cdot,\cdot)$ is almost surely H\"older continuous on any compact subset  of $[0,T]\times\mathbb{R}$, with  H\"older exponent $\frac{\gamma}{\alpha}$ in the temporal variable  and H\"older exponent $\gamma$ in the spatial variable.
\end{theorem}

To establish the   existence and uniqueness of the strong solution, we make the following assumption.
\begin{itemize}
\item[(H2)]  Assume that $\sigma(t,x,u)\in \mathcal{C}^{0,1,1}([0,T]\times\mathbb{R}^2)$ satisfies the following conditions: $\left|\sigma_u'(t,x,u)\right|$ and $\left|\sigma_{x,u}''(t, x, u)\right|$ are uniformly bounded, i.e., there exists a constant $C>0$ such that
\begin{equation}\label{eq sigma 3}
\sup_{t\in[0,T],\, x\in\mathbb{R}}\left|\sigma_u'(t, x, u)\right|\le C;
\end{equation}
\begin{equation}\label{eq sigma 4}
\sup_{t\in[0,T],\, x\in\mathbb{R}}\left|\sigma_{x, u}''(t, x, u)\right|\le C.
\end{equation}
 Moreover, for some $p>\frac{2(\alpha+1)}{4H-3+\alpha}$,
 \begin{equation}\label{eq sigma 5}
\sup_{t\in[0,T],\, x\in\mathbb{R}}\lambda^{-\frac1p}(x)\left|\sigma_u'(t,x,u_1)-\sigma_u'(t,x,u_2)\right|\le C|u_1-u_2|.
\end{equation}
\end{itemize}

\begin{theorem}\label{thm main2}  Assume that $\sigma(t,x,u)$ satisfies hypothesis (H2) and   that,  for some   $p>\frac{2(\alpha+1)}{4H-3+\alpha}$,    the   initial datum  $u_0$ belongs to $\mathcal{Z}_{\lambda,0}^p$.
 Then \eqref{SFHE} admits a unique strong solution  whose  sample paths lie in $\mathcal{C}([0,T]\times\mathbb{R})$ almost surely. Moreover,  the process $u(\cdot,\cdot)$ is almost surely uniformly H\"older continuous   on any compact set  in $[0,T]\times\mathbb{R}$, with the same temporal and spatial H\"older exponents as those in Theorem \ref{thm main1}.
 \end{theorem}

  The paper is organized as follows. Section 2 provides estimates of the first and second order differences of the fractional heat kernel, including the interaction between the weight $\lambda(x)$ and the fractional heat kernel $G_{\alpha}(t,x)$. Section 3 contains some preliminaries on stochastic integration with respect to the noise $W$, along with the basic moment estimates and H\"older continuity properties of stochastic convolutions. In Section 4, we establish the existence and uniqueness of the solution via an approximation argument.

  Throughout this paper, for two functions $f$ and $g$, the notation $f \lesssim g$ means that there exists a positive constant $c_{p,H,\alpha,T}$, which may depend on $p$, $H$, $\alpha$, and $T$, such that $f \le c_{p,H,\alpha,T}  \, g$. The notation $f \simeq g$ indicates  that both $f \lesssim g$ and $g \lesssim f$ hold.

\section{Properties of the fractional heat kernel}\label{subsec Aux}

In this section, we first recall some properties of the heat kernel  $G_{\alpha}(t,x)$ associated with  the fractional  Laplacian $-(-\Delta)^{ \alpha/2}$, and then derive   estimates of its first and second order difference, including   the interaction between   $\lambda(x)$ and the heat kernel $G_{\alpha}(t,x)$.

\subsection{The fractional  heat kernel $G_{\alpha}$}

The heat kernel $\{G_{\alpha}(t,  x)\}_{t>0,x\in \mathbb R}$ is defined via its Fourier transform
\begin{equation}
( \mathscr {F}G_{\alpha}(t,\cdot))(\xi)=e^{-t|\xi|^{\alpha}}, \qquad \xi\in\mathbb{R},
\end{equation}
for $\alpha\in(1,2)$; see,  e.g.,  \cite{BG1960, normalization, CZ2016}. It is well known that   $G_{\alpha}(t, \cdot)$ is the probability transition density function of a    $1$-dimensional    stable process $\{L_t^{\alpha}\}_{t\ge0}$, and   $G_{\alpha}(t, x)$ satisfies the following scaling property:
   \begin{equation}\label{eq scaling}
    G_{\alpha}(t, x)=t^{-\frac{1}{\alpha}} G_{\alpha}\big(1, t^{-\frac{1}{\alpha}}x\big) \ \ \ \  \ \ \ (t>0,\, x\in \mathbb R).
  \end{equation}
  
  According to \cite[Theorem 1.1]{CZ2016}, we have the following estimates.
     \begin{lemma}
  \begin{itemize}
  \item[(a)]  There exist   finite positive constants $c_{1}$ and $c_{2}$
 such that  for all $t>0$ and $x \in \mathbb R$,
\begin{align}\label{eq Green 3}
c_{1}t\left(t^{1/\alpha}+ |x| \right)^{-1-\alpha}\le G_{\alpha}(t, x)\le c_{2}t\left(t^{1/\alpha}+|x| \right)^{-1-\alpha}.
\end{align}

\item[(c)]   There exists a positive constant $c>0$ such that  for all $t>s>0$  and $x\in \mathbb R$,
 \begin{equation}\label{eq grad temp}
\begin{split}
 \left|G_{\alpha}(t, x) -G_{\alpha}(s, x) \right|
\le  &\,  c \frac{(t-s)}{s} G(s, x).
\end{split}
\end{equation}
  \end{itemize}
    \end{lemma}
  
 For each  $k\in \mathbb{N}$, let $\nabla^k $ stand for the $k$-order gradient with respect to the spatial variable $x$.  According to \cite[Lemma 2.2]{CZ2016}, we have the following results. 
      
     \begin{lemma}\label{lem alpha esti1}  
  \begin{itemize}
\item[(a)]  For each $\alpha\in (1,2)$, $k\in \mathbb N$, there exists a constant $c>0$ such that for all $t>0$, $x\in \mathbb R$, 
  \begin{equation}\label{alpha grad 1}
\left|\nabla^k G_{\alpha} (t, x)\right| \le  c  t\left(t^{1/{\alpha}}+|x|\right)^{-1-\alpha-k}.
\end{equation}
 \item[(b)]
 For any $\alpha\in(1,2)$,   there exists a positive constant $c$  such that  for all $t>0, h\in\mathbb{R}$,
\begin{equation}\label{eq grad spat}
\left|G_{\alpha}(t,x+h)-G_{\alpha}(t,x)\right|\le c  \left(\frac{|h|}{t^{\frac{1}{\alpha}}}\wedge  1 \right) \left(G_{\alpha}(t,x+h)+G_{\alpha}(t,x)\right).
\end{equation}
In particular, when $t=1$,  
 \begin{equation}\label{eq grad spat1}
\left|G_{\alpha}(1, x+h)-G_{\alpha}(1, x)\right|\le c_{\alpha} \begin{cases}
|h|G_{\alpha}(1,x),&\  |h|\le 1;\\G_{\alpha}(1,x+h)+G_{\alpha}(1,x),&\  |h|> 1.

\end{cases}
\end{equation}
    \end{itemize}
   \end{lemma}
 
\begin{proof} All the above  results,  except for the first inequality in \eqref{eq grad spat1},  are given in \cite[Lemma 2.2]{CZ2016}.   For completeness, we provide the proof for \eqref{eq grad spat1} in the case  $|h|\le 1$. 

Note that 
   \begin{align*} 
   G_{\alpha}(1,x+h)-G_{\alpha}(1,x)&=\int_0^1 \frac{d}{ds}G_{\alpha}(1, x+s h)=h\int_0^1 \nabla G_{\alpha}(1, x+s h) ds.
\end{align*}
By \eqref{alpha grad 1} and \eqref{eq Green 3}, we have 
 \begin{align*}
 \left| G_{\alpha}(1, x+h)-G_{\alpha}(1, x) \right|
  \lesssim &\,  |h|\int_0^1\frac{1}{(1+|x+sh|)^{2+\alpha}}ds\\ 
\lesssim &\, |h|\frac{1}{(1+|x|)^{2+\alpha}}\\
\lesssim &\,  |h|G_{\alpha}(1,x),
\end{align*}
where   the elementary inequality $1+|x| \le   2 (1+|x+z|)$ for all $|z|\le 1$ is used in the last  second  inequality.
   The proof is complete.
\end{proof}

  \subsection{The first and second order differences of $G_{\alpha}$}
     
 As in \cite{HW2022}, we investigate the following two increments related to the fractional  heat kernel $G_{\alpha}$.
\begin{itemize}
\item[(i)] The first order difference:
 \begin{equation}\label{nota alpha  1} 
 D(t,x,h):= \, G_{\alpha}(t,x+h)-G_{\alpha}(t,x),  
  \end{equation}
\item[(ii)] The second order difference:
 \begin{equation}\label{nota alpha 2}
\Box(t,x,y,h) :=  \, D(t,x+y,h)-D(t,x,h).
\end{equation}
Particularly, when $t=1$, we denote
   \begin{equation}\label{nota alpha  time 1}
 \begin{split}
    D(x,h):= &\, D(1,x,h), \\
   \Box(x,y,h):=&\, \Box(1,x,y,h).
  \end{split}\end{equation}
   \end{itemize}
  
\begin{lemma}\label{lem 2.8}
For any $\beta,\gamma\in (0,1)$, we have
\begin{align} 
\int_{\mathbb{R}^2}|D(t,x,h)|^2|h|^{-1-2\beta}dhdx=&\,  c_{\alpha,\beta} t^{-\frac{1+2\beta}{\alpha}}, \label{eq  2.17}\\
  \int_{\mathbb{R}^3}|\Box(t,x,y,h)|^2|h|^{-1-2\beta}|y|^{-1-2\gamma}dydhdx=&\, c_{\alpha,\beta,\gamma}t^{-\frac{2\beta+2\gamma+1}{\alpha}}. \label{eq 2.18}
\end{align}
\end{lemma}
\begin{proof} 
By the scaling property  \eqref{eq scaling},   for any $t>0$ and $x, y, h\in\mathbb{R}$,
\begin{equation}\label{eq scaling 2}
\begin{split}
D(t,x,h)=&\, t^{-\frac{1}{\alpha}}D\left(t^{-\frac{1}{\alpha}}x,t^{-\frac{1}{\alpha}}h\right),\\
 \Box(t,x,y,h)=&\, t^{-\frac{1}{\alpha}}\Box\left(t^{-\frac{1}{\alpha}}x,t^{-\frac{1}{\alpha}}y,t^{-\frac{1}{\alpha}}h\right).
 \end{split}
\end{equation}
Using changes of variables,   to prove this lemma  it suffices to show that 
\begin{equation}\label{eq iden 1}
\begin{split}
&\int_{\mathbb{R}^2}|D(x,h)|^2|h|^{-1-2\beta}dhdx<\infty;\\
&\int_{\mathbb{R}^3}|\Box(x,y,h)|^2|h|^{-1-2\beta}|y|^{-1-2\gamma}dydhdx<\infty.
\end{split}
\end{equation}

Note that the Fourier transforms of $D(x, h)$ and $\Box(x, y, h)$ with respect to $x$ are given respectively by 
\begin{align*}
 \mathscr {F}[D(\cdot, h)](\xi)=&\, e^{-|\xi|^{\alpha}}\left(e^{ih\xi}-1\right),\\
 \mathscr {F}[\Box(\cdot, y, h)](\xi)=&\, e^{-|\xi|^{\alpha}}\left(e^{ih\xi}-1\right)\left(e^{iy\xi}-1\right).
\end{align*}
  Thus, by Parseval's identity, 
  \begin{align*}
\int_{\mathbb{R}}|D(x,h)|^2dx =& \, \int_{\mathbb{R}}e^{-2|\xi|^{\alpha}}(1-\cos(h\xi))d\xi,\\  
\int_{\mathbb{R}}|\Box(x,y,h)|^2dx  = & \, \int_{\mathbb{R}}e^{-2|\xi|^{\alpha}}(1-\cos(h\xi))(1-\cos(y\xi))d\xi.
\end{align*}
By Fubini's theorem and a change of variables, for any $\beta\in (0,1)$,
\begin{align*}
\int_{\mathbb{R}^2}|D(x, h)|^2|h|^{-1-2\beta}dhdx=&\, \int_{\mathbb{R}}e^{-2|\xi|^{\alpha}}d\xi
\int_{\mathbb{R}}(1-\cos(h\xi))|h|^{-1-2\beta}dh\\
=&\, \int_{\mathbb{R}}e^{-2|\xi|^{\alpha}}|\xi|^{2\beta}d\xi
\int_{\mathbb{R}}(1-\cos(h))|h|^{-1-2\beta}dh<\infty.
\end{align*} 
Similarly,  for any $\beta, \gamma \in (0,1)$,  
 \begin{align*}
& \int_{\mathbb{R}^3}|\Box(x, y, h)|^2|h|^{-1-2\beta} |y|^{-1-2\gamma}dy dx dh\\
=&\, \int_{\mathbb{R}}e^{-2|\xi|^{\alpha}}d\xi \int_{\mathbb{R}}(1-\cos(h\xi))|h|^{-1-2\beta}dh \int_{\mathbb{R}}(1-\cos(y\xi))|y|^{-1-2\gamma}dy\\
 =&\, \int_{\mathbb{R}}e^{-2|\xi|^{\alpha}}|\xi|^{2(\beta+\gamma)}d\xi
\int_{\mathbb{R}}(1-\cos(h))|h|^{-1-2\beta}dh \int_{\mathbb{R}}(1-\cos(y))|y|^{-1-2\gamma}dy <\infty.
\end{align*} 
  
  The proof is complete. 
\end{proof}
 
 \begin{lemma}\label{prop alpha esti1} Recall $D(x,h)$ defined in \eqref{nota alpha  1}.  For $0<H<\frac{1}2$,  there exists a positive finite constant $c_{H,\alpha}$ depending on $H$ and $\alpha$ such that for any $t>0$  and $x\in\mathbb{R}$,
\begin{equation}\label{eq alpha esti1-1}
\int_{\mathbb{R}} \left|D(t,x,h)\right|^2|h|^{2H-2}dh\le c_{H,\alpha}\left (t^{\frac{2H-3}{\alpha}}\wedge\frac{|x|^{2H-2}}{t^{\frac1{\alpha}}}\right).
\end{equation}
\end{lemma}
\begin{proof}
By \eqref{eq scaling 2}, it suffices to show that 
 \begin{equation}\label{eq alpha esti1}
\int_{\mathbb{R}} \left|D(x,h)\right|^2|h|^{2H-2}dh\le c_{H,\alpha}\left (1\wedge|x|^{2H-2}\right).
\end{equation}

Without loss of  generality,  we assume $x>0$.  By Lemma \ref{lem alpha esti1},
\begin{align*}
& \int_{\mathbb{R}}  \left|D(x,h)\right|^2|h|^{2H-2}dh\\
=&\, \int_{|h|\le 1} \left|D(x,h)\right|^2|h|^{2H-2}dh+\int_{|h|>1} \left|D(x,h)\right|^2|h|^{2H-2}dh\\
 \lesssim &\,  \int_{|h|\le 1} G_{\alpha}(1,x)^2|h|^{2H}dh+\int_{|h|>1} G_{\alpha}(1,x)^2|h|^{2H-2}dh+\int_1^{\infty}  G_{\alpha}(1,x+h)^2 h^{2H-2}dh\\
 & \,\,\, + \int_{-\infty}^{-1}  G_{\alpha}(1,x+h)^2 |h|^{2H-2}dh\\
 =: &\, I_1+I_2+I_3+I_4.
 \end{align*}

By \eqref{eq Green 3}, we have
\begin{align}
I_1+ I_2+I_3 \lesssim (1+|x|)^{-2-2\alpha} \lesssim (1+|x|)^{2H-2}.
\end{align}
Since  $G_{\alpha}(1, x)$ is bounded,   
to prove \eqref{eq alpha esti1}, it suffices to show that there exists a positive constant $c_{H,\alpha}$ such that 
  \begin{equation}\label{eq midd1}
  I_4\le c_{H,\alpha} |x|^{2H-2} \ \   \text{for  any } x\in \mathbb R. 
   \end{equation}
     Note that 
\begin{align*}
I_4= &\,  \int_1^{\infty}  G_{\alpha}(1,x-h)^2 h^{2H-2}dh\\
 \lesssim &\,    \int_{\frac12x}^{2x}G_{\alpha}(1,x-h)^2 |h|^{2H-2}dh+\int_{[1,\infty)\cap[\frac{x}{2},2x]^c}G_{\alpha}(1,x-h)^2 |h|^{2H-2}dh\\
 =: &\,  I_{4,1}+I_{4,2}.
\end{align*}
When  $\frac12x<h<2x$,   we have
\begin{align*}
I_{4,1} \lesssim \, |x|^{2H-2}\int_{\frac{x}{2}}^{2x} G_{\alpha}(1,x -h)^2dh\lesssim |x|^{2H-2}\int_{\mathbb{R}} G_{\alpha}(1,x-h)^2dh\lesssim |x|^{2H-2}.
\end{align*}
If $h>2x$ or $h<\frac12x$, then $|h-x|\ge \frac12|x|$. Consequently,
\begin{align*}
I_{4,2} \lesssim&\, \int_{[1,\infty)\cap[\frac{x}{2},2x]^c}\frac{1}{(1+|h-x|)^{2+2\alpha}} |h|^{2H-2}dh \\
 \lesssim&\, \int_{[1,\infty)\cap[\frac{x}{2},2x]^c}\frac{1}{(1+|x|)^{2+2\alpha}} |h|^{2H-2}dh \\
 \lesssim&\, \frac{1}{(1+|x|)^{2+2\alpha}}.
\end{align*} 
 
 The proof is complete.
\end{proof}

Similarly to the proof of   Lemma \ref{lem alpha esti1}, we have the following result.
\begin{lemma}\label{lem alpha esti2}
Recall $\Box(x,y,h)$ defined in \eqref{nota alpha 2}. For any $0<H<\frac12$ and $1<\alpha<2$,  there exists a positive constant $c_\alpha$ such that
\begin{equation}\label{eq alpha esti0}
\left|\Box(x,y,h)\right|\le c_{\alpha}
\begin{cases}
|y||h|\left(\frac{1}{1+|x|}\right)^{3+\alpha}, &|y|\le1, |h|\le 1;\\ 
|D(x+y,h)|+|D(x,h)|,& |y|>1,\text{ or } |h|>1.
\end{cases}
\end{equation}
  \end{lemma}
\begin{proof}  
 The result for the cases $|y| > 1$ or $|h| > 1$ follows directly from the triangle inequality.
We now prove \eqref{eq alpha esti0} in the remaining case $|y| \le 1$ and $|h| \le 1$ using \eqref{alpha grad 1}.
 
    For any $s, t\in [0,1], x\in \mathbb{R}$, set $ \gamma(s,t):=x+sy+th$.
 Then, 
  \begin{equation}\label{eq alpha differ2}
  \begin{split}
\Box(x,y,h) =&\, D(x+y,h)-D(x,h) \\
 =&\, \int_0^1\frac{d}{ds}\left[G_{\alpha}(1,\gamma(s,1))- G_{\alpha}(1,\gamma(s,0))\right]ds \\
 =&\, \int_0^1\int_0^1\frac{d}{dt}\frac{d}{ds}G_{\alpha}(1,\gamma(s,t))dsdt \\
  =&\,  yh \int_0^1\int_0^1\nabla^2G_{\alpha}(1,\gamma(s,t)) dsdt. 
 \end{split}
 \end{equation}
By \eqref{alpha grad 1}, we have
 \begin{equation}\label{eq green o2}
 \left|\nabla^2G_{\alpha}\right|(1,\gamma(s,t))\lesssim \left(\frac{1}{1+|\gamma(s,t)|}\right)^{3+\alpha}\lesssim \left(\frac{1}{1+|x|}\right)^{3+\alpha},
 \end{equation}
 where we use   the elementary inequality $\frac{1+|x|}{1+|x+z|} \le 3$ for $|z|\le 2$.
 
 Combining  \eqref{eq alpha differ2} and \eqref{eq green o2} yields
  \begin{align*}
 \left|\Box(x,y,h)\right|&\lesssim |y||h|\left(\frac{1}{1+|x|}\right)^{3+\alpha}.
  \end{align*}
  
     The proof is complete. 
 \end{proof}

\begin{lemma}\label{lem integ} For any $H\in (0,\frac12)$, there exists a constant  $c_H>0$ depending on $H$ such that for any $x\in\mathbb{R}$,
 \begin{equation}\label{eq integ aim0}
 \int_{|y|>1}|y|^{2H-2} \left(1\wedge |x+y|^{2H-2}\right)dy\le c_H\left(1\wedge |x|^{2H-2}\right).
\end{equation}
\end{lemma}
\begin{proof}
Since
 \begin{equation*}
 \int_{|y|>1}|y|^{2H-2} \left(1\wedge |x+y|^{2H-2}\right)dy\le \int_{|y|>1}|y|^{2H-2}dy=\frac{2}{1-2H},
\end{equation*} 
it suffices to show that for any $x>2$,
\begin{equation}\label{eq integ 1}
\int_{1}^{\infty}|y|^{2H-2}\left(1\wedge |x+y|^{2H-2}\right)dy\lesssim   |x|^{2H-2},
\end{equation}
\begin{equation}\label{eq integ 2}
\int_{1}^{\infty}|y|^{2H-2}\left(1\wedge |x-y|^{2H-2}\right)dy \lesssim   |x|^{2H-2}.
\end{equation}

Estimate \eqref{eq integ 1} follows easily from 
\begin{align*}
\int_{1}^{\infty}|y|^{2H-2}\left(1\wedge |x+y|^{2H-2}\right) dy \le \int_{1}^{\infty}|y|^{2H-2} x^{2H-2}dy\lesssim x^{2H-2}.
\end{align*}

It remains to prove \eqref{eq integ 2}  for any $x>2$. We decompose the integral as 
\begin{equation}\label{eq integ 3}
\begin{split}
&\int_{1}^{\infty}|y|^{2H-2}\left(1\wedge |x-y|^{2H-2}\right)dy\\
 =&\, \int_{1}^{\frac{x}2}y^{2H-2}\left(1\wedge |x-y|^{2H-2}\right)dy+\int_{\frac{x}{2}}^{\infty}y^{2H-2}\left(1\wedge |x-y|^{2H-2}\right)dy  \\
 =:&\, I_1+I_2.  
\end{split}
\end{equation}

For  $y\in [1,\frac{x}2]$,  we have $x-y\ge \frac{x}{2}>1$.  Hence, 
 \begin{align}\label{eq integ 31}
{I_1}\le4 x^{2H-2}\int_1^{\frac{x}2}y^{2H-2}dy\lesssim x^{2H-2}.
\end{align}

For    $y\in [\frac{x}2,\infty]$, we have   $y^{2H-2}\le 4x^{2H-2}$.  Hence, 
\begin{equation}\label{eq integ 32}
\begin{split}
{I_2} \le &\,  4 x^{2H-2}\int_{\frac{x}2}^{\infty}\left(1\wedge |x-y|^{2H-2}\right)dy\\
\le &\, 4x^{2H-2} \left(\int_{x-1}^{x+1}1dy+\int_{[\frac{x}2,x-1]\cup [x+1,\infty]}|x-y|^{2H-2}dy\right)\\
\le &\,  4x^{2H-2}\left(2+2\int_{z\ge1} z^{2H-2}dz\right)\\
 \lesssim &\, x^{2H-2},
 \end{split}
\end{equation}
Putting \eqref{eq integ 3}-\eqref{eq integ 32} together, we obtain \eqref{eq integ 2}.
 The proof is complete.  
  \end{proof}

\begin{lemma}\label{prop alpha esti2}
For any $0<H<\frac12$ and $1<\alpha<2$, there exists a positive constant $c_{H,\alpha}$ such that for all $t>0$ and $x\in \mathbb R$,
\begin{equation}\label{eq alpha esti2-1}
\int_{\mathbb{R}^2 }\left|\Box(t,x,y,h)\right|^2|h|^{2H-2}|y|^{2H-2}dydh\le c_{H,\alpha}\left(t^{\frac{4H-4}{\alpha}}\wedge \frac{|x|^{2H-2}}{t^{\frac{2-2H}{\alpha}}}\right).
\end{equation}
\end{lemma}
\begin{proof}
By the scaling property  \eqref{eq scaling 2}, it suffices to prove 
\begin{equation}\label{eq alpha esti2}
\int_{\mathbb{R}^2 }\left|\Box(x,y,h)\right|^2|h|^{2H-2}|y|^{2H-2}dydh\le c_{H,\alpha} \left(1\wedge |x|^{2H-2}\right).
\end{equation} 
 Define the following two regions:
\begin{align*}
A:=&\, \{(y,h): |y|\le 1,|h|\le 1\},\\
\bar A:=&\,\{(y,h): |y|> 1  \text{ or } |h|> 1\}.
\end{align*}
 
 When $|y|\le 1$ and  $|h|\le 1$,  using the first estimate in \eqref{eq alpha esti0},
\begin{align*}
\int_{A}\left|\Box(x,y,h)\right|^2|h|^{2H-2}|y|^{2H-2}dydh \lesssim &\,  \int_{A}|y|^{2H}|h|^{2H}\left(\frac{1}{1+|x|}\right)^{6+2\alpha}dydh\\
 \lesssim &\, \left(\frac{1}{1+|x|}\right)^{6+2\alpha}.
\end{align*}
 
 Using   \eqref{eq alpha esti0} and  Lemma \ref{prop alpha esti1}, we have
\begin{align*}
 \int_{\bar A}\left|\Box(x,y,h)\right|^2|h|^{2H-2}|y|^{2H-2}dydh 
\le  &\, 
 \int_{|y|>1}|y|^{2H-2}dy\int_{\mathbb R} |D(x+y,h)|^2 |h|^{2H-2}dh\\
 &\,\,\,\, +\int_{|y|>1}|y|^{2H-2}dy\int_{\mathbb R} |D(x,h)|^2 |h|^{2H-2}dh\\
\lesssim &\,  \int_{|y|>1}|y|^{2H-2} \left(1\wedge |x+y|^{2H-2}\right)dy+\left(1\wedge |x|^{2H-2}\right)\\
\lesssim &\, \left(1\wedge |x|^{2H-2}\right),
\end{align*}
 where  \eqref{eq integ aim0} is used in the last inequality.  The proof is complete. 
\end{proof}

\subsection{Some estimates of the heat kernel on the weighted space}
 
For any $\lambda\in \mathbb R$,  define 
\begin{align}\label{eq lambda}
\lambda(x):= c(\lambda) \frac{1}{(1+|x|^2)^{\lambda}},
\end{align}
 where $c(\lambda)$ is a normalized constant satisfying $\int_{\mathbb{R}}\lambda(x)dx=1$.  
To avoid using too many notations, we use the symbol  $\lambda$
for both the real number and the induced function, as in Hu and Wang \cite{HW2022}.

To handle  the weight $\lambda(x)$, we need several  technical estimates concerning the interaction between $\lambda(x)$ and the heat kernel $G_{\alpha}(t,x)$.

\begin{lemma}\label{lem 2.5}
 For every $1<\alpha<2$ and  $q>\frac{1}{1+\alpha}$,  let $\lambda(x)$ be the function defined by  \eqref{eq lambda} with  $$\lambda\in \left(-\frac{(1+\alpha)q-1}{2},\frac{(1+\alpha)q-1}{2}\right).$$ 
  Then,  for any $t\in [0,T]$,
\begin{equation}\label{eq power esti2}
\sup_{x\in\mathbb{R}} \frac{1}{\lambda(x)}\int_{\mathbb{R}}G_{\alpha}(t,x-y)^q \lambda(y) dy\le c_{\lambda,q,\alpha,T} t^{-\frac{q-1}{\alpha}}.
\end{equation}
Particularly, taking $q=1$,  we obtain that for any  $\lambda\in(-\frac{\alpha}{2},\frac{\alpha}{2})$, 
\begin{equation}\label{eq power esti1}
\sup_{0\le t\le T}\sup_{x\in\mathbb{R}}\frac{1}{\lambda(x)}\int_{\mathbb{R}}G_{\alpha}(t,x-y)\lambda(y)dy<\infty.
\end{equation}
  \end{lemma}
\begin{proof}
  By the scaling property \eqref{eq scaling} and a change of variables,   for any   $t\le T$, we have   
  \begin{align*}
 \sup_{x\in \mathbb R}\int_{\mathbb{R}}G_{\alpha}(t,x-y)^q\frac{\lambda(y)}{\lambda(x)}dy 
 = &\,   t^{-\frac{q-1}{\alpha}}   \sup_{x\in \mathbb R}   \int_{\mathbb{R}} G_{\alpha}(1,z )^q\frac{\lambda\left(x+t^{\frac1{\alpha}}z\right)}{\lambda\left(x\right)}dz\\
 \le &\,   c_{\lambda,q,\alpha}t^{-\frac{q-1}{\alpha}}\int_{\mathbb{R}} \frac{1}{(1+|z|)^{(1+\alpha)q}}\left(1+t^{\frac1{\alpha}}|z|\right)^{2|\lambda|}dz\\
 \le  &\,   c_{\lambda,q,\alpha,T}t^{-\frac{q-1}{\alpha}}\int_{\mathbb{R}}(1+|z|)^{2|\lambda|-(1+\alpha)q}dz.
\end{align*} 
 Here, the first inequality uses the estimate 
  $$
\sup_{x\in\mathbb{R}}\frac{\lambda(x+y)}{\lambda(x)}\le c_{\lambda}(1+|y|)^{2|\lambda|},
$$
 (cf. \cite[Lemma 2.5]{HW2022}),  and the second inequality follows from  \eqref{eq Green 3}.   
The final integral is finite precisely when   $2|\lambda|<(1+\alpha)q-1$. 

The proof is complete. 
\end{proof}

Applying Lemma \ref{lem 2.5} with $\lambda=1-H$ and  $q=2$,   we obtain  the following result. 
\begin{corollary}
For  any $ \frac{2-\alpha}{2}<H<\frac 12$,     it holds that
 \begin{equation}\label{eq power esti3}
\sup_{x\in\mathbb{R}}\int_{\mathbb{R}}G_{\alpha}(t,x-y)^2\frac{(1+|y|^2)^{1-H}}{(1+|x|^2)^{1-H}}dy\le c_{H,\alpha,T} t^{-\frac{1}{\alpha}}.
\end{equation}
\end{corollary}

\begin{lemma}\label{lem 2.12} Assume $0<H<\frac12$ and $1<\alpha<2$. Denote $\lambda(x)=\frac{1}{(1+|x|^2)^{1-H}}$.  We have
\begin{align} 
&\int_{\mathbb{R}^2}\left|D(t,x,h)\right|^2|h|^{2H-2}\lambda(z-x)dxdh\le c_{T,H,\alpha}t^{\frac{2(H-1)}{\alpha}}\lambda(z), \label{eq 2.27}\\
&\int_{\mathbb{R}^3}\left|\Box(t,x,y,h)\right|^2|h|^{2H-2}|y|^{2H-2}\lambda(z-x)dxdydh\le c_{T,H,\alpha}t^{\frac{4H-3}{\alpha}}\lambda(z). \label{eq 2.27-2}
\end{align}
\end{lemma}
\begin{proof}
For any $x, z\in \mathbb R$,  set
$$
R(x,z):=\frac{\lambda(z-x)}{\lambda(z)}\simeq \left(\frac{1+|z|}{1+|x-z|}\right)^{2-2H}.
$$
By Lemma \ref{prop alpha esti1},  the scaling property \eqref{eq scaling},   and the changes of variables $x\to xt^{\frac1{\alpha}}$ and $h\to ht^{\frac1{\alpha}}$, we have
\begin{equation*}
\begin{split}
& \int_{\mathbb{R}^2}\left|D(t,x,h)\right|^2|h|^{2H-2}R(x,z)dxdh \\
=&\, t^{\frac{2(H-1)}{\alpha}}\int_{\mathbb{R}^2}\left|D(x,h)\right|^2|h|^{2H-2}R\left(t^{\frac1{\alpha}}x, z\right)dxdh \\
 \le &\, c_{H,\alpha}t^{\frac{2(H-1)}{\alpha}}\int_{\mathbb{R}}1\wedge |x|^{2H-2} R\left(t^{\frac1{\alpha}}x, z\right)dx.
\end{split}
\end{equation*}
 According to  \cite[(2.30)]{HW2022}, we have
$$
\sup_{t\in[0,T]}\sup_{z\in\mathbb{R}}\int_{\mathbb{R}}1\wedge |x|^{2H-2} R\left(t^{\frac1{\alpha}}x, z\right)dx<\infty.
$$
Thus, the first estimate \eqref{eq 2.27} follows.

 Similarly, using  Lemma \ref{prop alpha esti2}, the scaling property \eqref{eq scaling}, and a change of variables, we get
\begin{equation*}
\begin{split}
&\int_{\mathbb{R}^3}\left|\Box(t,x,y,h)\right|^2|h|^{2H-2}|y|^{2H-2}R(x,z)dxdydh \\
 =&\, t^{\frac{4H-3}{\alpha}}\int_{\mathbb{R}^3}\left|\Box(x,y,h)\right|^2|h|^{2H-2}|y|^{2H-2}R\left(t^{\frac1{\alpha}}x, z\right)dxdydh \\
 \le &\,  c_{H,\alpha}t^{\frac{4H-3}{\alpha}}\int_{\mathbb{R}}1\wedge |x|^{2H-2} R\left(t^{\frac1{\alpha}}x, z\right)dx \\
 \lesssim &\, t^{\frac{4H-3}{\alpha}}. 
\end{split}
\end{equation*}
This proves \eqref{eq 2.27-2} and completes the proof.
\end{proof}

\section{Some bounds for stochastic convolutions}\label{sec-3}

\subsection{Stochastic integral}

 In this section, we   recall the stochastic integral with respect to the Gaussian noise $\dot W$ and the definitions of the solutions, borrowed from \cite{HHLNT2017} and \cite{HW2022}.
  
 Denote by  $ \mathcal{D}=\mathcal{D}(\mathbb{R}) $  the space of   {real-valued infinitely differentiable} functions with compact support on $\mathbb{R}$. 
 The Fourier transform of a function $f\in\mathcal{D}$ is defined as
 $$
   \mathscr {F} f(\xi)=\int_{\mathbb{R}} e^{-i\xi x}f(x) dx.
$$

 Let   $(\Omega,\mathcal{F},\mathbb{P})$ be a complete probability space.  Let
    $\mathcal{D}(\mathbb{R}_{+}\times \mathbb{R})$  be  the space of real-valued infinitely differentiable functions with compact support on $\mathbb{R}_{+}\times \mathbb{R}$.
The noise $\dot W$ is  a zero-mean Gaussian family $\{W(\phi), \phi\in\mathcal{D}(\mathbb{R}_{+}\times \mathbb{R})\}$  with the covariance structure given by
 \begin{equation}\label{CovStru}
   \mathbb{E}\big[ W(\phi)W(\psi)\big]=c_{H}\int_{\mathbb{R}_{+}\times \mathbb{R}} \mathscr {F} \phi(s, \cdot)(\xi) \overline{ \mathscr {F}\psi(s, \cdot)(\xi)}\cdot |\xi|^{1-2H}d\xi ds,
 \end{equation}
 where  $c_{H}$ is given by   \eqref{e.c1},   and $ \mathscr {F}\phi(s, \cdot)(\xi)$ is the Fourier transform with respect to the spatial variable $x$ of the function $\phi(s,x)$.
 Let $\mathcal F_t$ be the filtration generated by $W$, namely
 $$
 \mathcal F_t=\sigma\left\{ W\big(\varphi(x) {\mathbf 1}_{[0,r]}(s)\big):\ r\in [0, t], \ \varphi\in \mathcal D(\mathbb R) \right\}.
 $$
  Equation \eqref{CovStru}  defines   a Hilbert scalar product on $\mathcal{D}(\mathbb{R}_{+}\times \mathbb{R}) $.
Denote $\mathfrak{H}$ the Hilbert space obtained by completing $\mathcal{D}(\mathbb{R}_{+}\times \mathbb{R}) $ with respect to this scalar  product.
\begin{proposition}  (\cite[Proposition 2.1, Equation (2.8)]{HHLNT2017}, \cite[Theorem 3.1]{PT2000})\label{hSpaceProp}
 The  space $\mathfrak{H}$ is a Hilbert space equipped with the scalar  product
\begin{equation*}
\begin{split}
     \langle\phi,\psi\rangle_{\mathfrak{H}}
     :=&\,c_{H}\int_{\mathbb{R}_{+}}\left(\int_ {\mathbb{R}} \mathscr {F} \phi(t,\xi) \overline{ \mathscr {F}\psi(t,\xi)}  |\xi|^{1-2H}d\xi \right)dt
     \\
     =&\, H\left(\frac12-H\right) \int_{\mathbb{R}_+}\left(\int_{\mathbb{R}^2}[\phi(t, x+y)-\phi(t, x)] [\psi(t, x+y)-\psi(t, x)] |y|^{2H-2} dxdy\right)dt.
   \end{split}
   \end{equation*}
 The space  $\mathcal{D}(\mathbb{R}_{+}\times \mathbb{R})$ is dense in $\mathfrak{H}$.
 \end{proposition}

We recall the stochastic integral  with respect to the rough noise $W$, borrowed from \cite{HHLNT2017}.
\begin{definition}(\cite[Definition 2.2]{HHLNT2017})\label{def StoIt}
An elementary process $g$ is a process given by
$$
    g(t,x)=\sum_{i=1}^{n}\sum_{j=1}^{m} X_{i,j}\textbf{1}_{(a_i,b_i]}(t)\textbf{1}_{(h_j,l_j]}(x),
$$
   where $n$ and $m$ are finite positive integers, $0\leq a_1<b_1<\cdots<a_n<b_n<\infty$, $h_j<l_j$ and $X_{i,j}$ are $\mathcal{F}_{a_i}$-measurable random variables for $i=1,\dots,n$, $j=1,\dots,m$. The stochastic integral of such  an elementary  process with respect to $W$ is defined as
   \begin{equation}\label{SI_Ele}
     \begin{split}
       \int_{\mathbb{R}_{+}}\int_{\mathbb{R}} g(t,x)W(dt,dx) 
     =&\, \sum_{i=1}^{n}\sum_{j=1}^{m}  X_{i,j}W(\textbf{1}_{(a_i,b_i]}\otimes \textbf{1}_{(h_j,l_j]}) \\
      =&\, \sum_{i=1}^{n}\sum_{j=1}^{m}  X_{i,j}\Big[W(b_i,l_j)-W(a_i,l_j)-W(b_i,h_j)+W(a_i,h_j)\Big].
     \end{split}
   \end{equation}
 \end{definition}

Hu et al. \cite[Proposition 2.3]{HHLNT2017} extend the definition of the integral with respect to $W$ to a broad class of adapted processes in the following way.
\begin{proposition}(\cite[Proposition 2.3]{HHLNT2017})\label{prop 2.3}
   Let $\Lambda_{H}$ be the space of predictable processes $g$ defined on $\mathbb{R}_{+}\times\mathbb{R}$ such that almost surely $g\in\mathfrak{H}$ and $\mathbb{E}[\|g\|_{\mathfrak{H}}^2]<\infty$.  Then,  the following items hold.
   \begin{itemize}
       \item[(i).]
  The space of the elementary processes
    defined in Definition \ref{def StoIt} is dense in $\Lambda_{H}$.
      \item[(ii).]
 For any  $g\in\Lambda_{H}$, the stochastic integral $\int_{\mathbb{R}_{+}}\int_{\mathbb{R}} g(s,x)W(ds,dx)$ is defined  as the $L^2(\Omega)$-limit of  Riemann sums along elementary processes approximating $g$
in $\Lambda_H$, and the following isometry property holds:
   \begin{equation}\label{eq iso}
\mathbb{E}\left[\left(\int_{\mathbb{R}_+\times\mathbb{R}}g(t,x)W(dt,dx)\right)^2\right]=\mathbb{E}\left[\|g\|^2_{\mathfrak{H}}\right].
\end{equation}
\end{itemize}
\end{proposition}

 %
%

 Let $(B,\|\cdot\|_{B})$ be a Banach space with  norm $\|\cdot\|_{B}$, and  let $\beta\in(0,1)$ be a fixed number. For any function $f:\mathbb{R}\to B$, define  
\begin{equation}\label{NBNorm}
   \mathcal{N}_{\beta}^{B}f(x):=\left(\int_{\mathbb{R}}\|f(x+h)-f(x)\|_B^2\cdot |h|^{-1-2\beta}dh\right)^{\frac 12},
 \end{equation}
whenever the  quantity is finite.  When $B=\mathbb{R}$, we abbreviate the notation $\mathcal{N}_{\beta}^{\mathbb{R}}f$  as  $\mathcal{N}_{\beta}f$.  As in \cite{HHLNT2017,HW2022},  when $B=L^p(\Omega)$, we  denote $ \mathcal{N}_{\beta}^{B}$ by $\mathcal{N}_{\beta,\,p}$; that is,
 \begin{equation}\label{NpNorm}
   \mathcal{N}_{\beta,\,p}f(x):=\left(\int_{\mathbb{R}}\|f(x+h)-f(x)\|^2_{L^p(\Omega)} \cdot |h|^{-1-2\beta}dh\right)^{\frac 12}.
 \end{equation}

The following Burkholder-Davis-Gundy inequality was obtained in  \cite{HHLNT2017}.
  \begin{proposition}\label{prop BDG}   $($\cite[Proposition 3.2]{HHLNT2017}$)$
   Let $W$ be the Gaussian noise with the covariance \eqref{CovW}, and let    $f\in\Lambda_H$  be a predictable random field. Then, we have  for any $p\ge 2$,
   \begin{equation}\label{eq BDG}
     \begin{split}
        \left\|\int_{0}^{t}\int_{\mathbb{R}}f(s,y) W(ds,dy)\right\|_{L^p(\Omega)}
     \leq \sqrt{4p}c_{H}\left(\int_{0}^{t}\int_{\mathbb{R}}\left[\mathcal{N}_{\frac 12-H,\,p}f(s,y)\right]^2dyds\right)^{\frac 12},
     \end{split}
   \end{equation}
   where $c_{H}$ is a constant depending only on $H$,  and    $\mathcal{N}_{\frac 12-H,\,p}f(s,y)$ denotes the application of $\mathcal{N}_{\frac 12-H,\,p} $  to the spatial variable $y$.
 \end{proposition}

\subsection{Some estimates for stochastic convolutions}

\begin{proposition}\label{prop 4.2}
For any $v\in \mathcal{Z}_{\lambda,T}^p$, define
\begin{equation}\label{eq stocha1}
\Phi(t,x)=\int_0^t\int_{\mathbb{R}}G_{\alpha}(t-s,x-y)v(s,y)W(ds,dy).
\end{equation}
Then  the following estimates hold.
\begin{itemize}
\item[(i).] If  $\frac{2-\alpha}{2}<H<\frac12$ and $p>\frac{2(\alpha+1)}{2H+\alpha-2}$, then
\begin{equation}\label{eq 4.8}
\left\|\sup_{t\in[0,T],\,x\in\mathbb{R}}\lambda^{\frac1{p}}(x)\Phi(t,x)\right\|_{L^{p}(\Omega)}\le c_{\alpha,T,p,H}\|v\|_{\mathcal{Z}_{\lambda,T}^p}.
\end{equation}
\item[(ii).] If $\frac{3-\alpha}{4}<H<\frac12$ and $p>\frac{2\alpha+2}{4H-3+\alpha}$, then
\begin{equation}\label{eq 4.9}
\left\|\sup_{t\in[0,T],\,x\in\mathbb{R}}\lambda^{\frac1{p}}(x)\mathcal{N}_{\frac12-H}\Phi(t,x)\right\|_{L^{p}(\Omega)}\le c_{\alpha,T,p,H}\|v\|_{\mathcal{Z}_{\lambda,T}^p}.
\end{equation}
\item[(iii).] If  $\frac{2-\alpha}{2}<H<\frac12$, $p>\frac{2(\alpha+1)}{2H+\alpha-2}$, and $0<\gamma<\frac{2H+\alpha-2}{2\alpha}-\frac{\alpha+1}{\alpha p},$ then
\begin{equation}\label{eq 4.10}
\left\|\sup_{t,\, t+h\in[0,T],x\in\mathbb{R}}\lambda^{\frac1{p}}(x)\left[\Phi(t+h,x)-\Phi(t,x)\right]\right\|_{L^{p}(\Omega)}\le c_{\alpha,T,p,H,\gamma}|h|^{\gamma}\|v\|_{\mathcal{Z}_{\lambda,T}^p}.
\end{equation}

\item[(iv).] If  $\frac{2-\alpha}{2}<H<\frac12$,  $p>\frac{2(\alpha+1)}{2H+\alpha-2}$, and $0<\gamma<\frac{2H+\alpha-2}{2}-\frac{\alpha+1}{p},$ then
\begin{equation}\label{eq 4.11}
\left\|\sup_{t\in[0,T],\,x\in\mathbb{R}}\frac{\Phi(t,x)-\Phi(t,y)}{\lambda^{-\frac1p}(x)+\lambda^{-\frac1p}(y)}\right\|_{L^{p}(\Omega)}\le c_{\alpha,T,p,H,\gamma}|x-y|^{\gamma}\|v\|_{\mathcal{Z}_{\lambda,T}^p}.
\end{equation}

 \end{itemize}
\end{proposition}

\begin{proof}  
Here, while we largely follow the framework of Proposition 4.2 in \cite{HW2022}, we need to deal with numerous additional challenges arising from the fractional heat kernel.
 
For any $\eta\in(0,1)$, set
 \begin{equation}\label{eq 4.12}
J(\eta,r,z):=\int_0^{r}\int_{\mathbb{R}}(r-s)^{-\eta}G_{\alpha}(r-s,z-y)v(s,y)W(ds,dy).
\end{equation}
A stochastic version of Fubini's theorem (see, e.g., \cite[Theorem 5.10]{DaPrato2014}) implies
\begin{equation}\label{eq 4.13}
\Phi(t,x)=\frac{\sin(\pi \eta)}{\pi}\int_0^{t}\int_{\mathbb{R}}(t-r)^{\eta-1}G_{\alpha}(t-r,x-z)J(\eta,r,z)dzdr.
\end{equation}

The first two steps are  devoted to proving Part (i).

  \noindent{\bf Step 1.}   In this step, we  obtain the desired growth estimate of $\Phi(t,x)$ in terms of $J(\eta,r,z)$. Assume that 
  \begin{equation}\label{eq p condi1}
  p>\frac{1-2H}{\alpha}.
    \end{equation}
    Taking $q=\frac{p}{p-1}$,  condition \eqref{eq p condi1}  is equivalent to $$\frac{2q(1-H)}{p}<(1+\alpha)q-1.$$
      By \eqref{eq power esti2} and   \eqref{eq 4.13}, we have
    \begin{align*}
   & \sup_{t\in[0,T],\,x\in\mathbb R} \lambda^{\theta}(x)\left|\Phi(t,x)\right|\\
     \simeq&\,  \sup_{t\in[0,T],\,x\in\mathbb R}\lambda^{\theta}(x)\left|\int_0^t\int_{\mathbb{R}}(t-r)^{\eta-1}G_{\alpha}(t-r,x-z)J(\eta,r,z)dzdr\right|\\
     \lesssim  &\,  \sup_{t\in[0,T],\,x\in\mathbb R}\lambda^{\theta}(x)\int_0^t(t-r)^{\eta-1}\left(\int_{\mathbb{R}}\left|G_{\alpha}(t-r,x-z)\lambda^{-\frac{1}{p}}(z)\right|^{q}dz\right)^{\frac{1}{q}}
    \left\|J(\eta,r,\cdot)\right\|_{L^p_{\lambda}(\mathbb{R})}dr\\
    \lesssim &\,   \sup_{t\in[0,T],\,x\in\mathbb R}\lambda^{\theta}(x)\int_0^t(t-r)^{\eta-1}(t-r)^{-\frac{q-1}{q\alpha}}\lambda(x)^{-\frac{1}{p}}\left\|J(\eta,r,\cdot)\right\|_{L^p_{\lambda}(\mathbb{R})}dr.
        \end{align*}
Setting $\theta=\frac1p$ and then applying the H\"older inequality, we  have
  \begin{equation}\label{eq 4.14}  
  \begin{split}
  \sup_{t\in[0,T],\,x\in\mathbb R}\lambda^{\theta}(x)\left|\Phi(t,x)\right| 
\lesssim &\,  \sup_{t\in[0,T]}\int_0^t(t-r)^{\eta-\frac{\alpha+1}{\alpha}+\frac{1}{q\alpha}}\left\|J(\eta,r,\cdot)\right\|_{L^p_{\lambda}(\mathbb{R})}dr \\
    \lesssim  &\,   \sup_{t\in[0,T]}\left(\int_0^t(t-r)^{q\eta-\frac{q(\alpha+1)}{\alpha}+\frac{1}{\alpha}}dr\right)^{\frac1q}
   \left(\int_0^T\left\|J(\eta,r,\cdot)\right\|_{L^p_{\lambda}(\mathbb{R})}^pdr\right)^{\frac1p} \\
    \lesssim &\,   \left(\int_0^T\left\|J(\eta,r,\cdot)\right\|_{L^p_{\lambda}(\mathbb{R})}^pdr\right)^{\frac1p}, 
    \end{split}
    \end{equation}
   which is finite provided that 
   \begin{equation}\label{eq p condi2}
   \eta>\frac{\alpha+1}{p\alpha}.
   \end{equation}
   This is possible when $p>\frac{\alpha+1}{\alpha}$. In that case, condition \eqref{eq p condi1}  follows immediately because $\alpha+2H>2$.  Thus, to prove Part (i), it suffices to show that there exists a constant $c>0$, independent of $r\in[0,T]$, such that
   \begin{equation}\label{eq 4.16}
   \mathbb{E}\left\|J(\eta,r,\cdot)\right\|_{L^p_{\lambda}(\mathbb{R})}^p\le c\|v\|_{\mathcal{Z}_{\lambda,T}^p}^p.
   \end{equation}

\noindent{\bf Step 2.} In this step,  we prove the   bound \eqref{eq 4.16}. Define
$$
\mathcal{D}_{1}(r,z):=\left(\int_0^r\int_{\mathbb{R}^2}(r-s)^{-2\eta}|D(r-s,y,h)|^2\|v(s,y+z)\|_{L^p(\Omega)}^2|h|^{2H-2}dhdyds\right)^{\frac{p}{2}},
$$ 
$$
\mathcal{D}_{2}(r,z):=\left(\int_0^r\int_{\mathbb{R}^2}(r-s)^{-2\eta}|G_{\alpha}(r-s,y)|^2\|\Delta_hv(s,y+z)\|_{L^p(\Omega)}^2|h|^{2H-2}dhdyds\right)^{\frac{p}{2}},
$$
where $\Delta_hv(t,x):=v(t,x+h)-v(t,x)$. 

By \eqref{eq 4.12}  and  Burkholder-Davis-Gundy's inequality \eqref{eq BDG}, we have
\begin{align*}
\mathbb{E}\left\|J(\eta,r,\cdot)\right\|_{L^p_{\lambda}(\mathbb{R})}^p\lesssim &\, \int_{\mathbb{R}}\Bigg\{\int_0^r\int_{\mathbb{R}^2} (r-s)^{-2\eta}\Big[\mathbb{E}\big|G_{\alpha}(r-s,y+h-z)v(s,y+h)\\
&\hskip 12pt-G_{\alpha}(r-s,y-z)v(s,y)\big|^{p}\Big]^{2/p}|h|^{2H-2}dhdyds\Bigg\}^{p/2}\lambda(z)dz\\
 \lesssim& \,  \int_{\mathbb{R}}\left[\mathcal{D}_1(r,z)+\mathcal{D}_2(r,z)\right]\lambda(z)dz\\
=: &\, D_1+D_2.
\end{align*}

For the first term $D_1$, by Minkowski's inequality,    \eqref{eq  2.17}, and \eqref{eq alpha esti1-1}, we have that  for any  $p>2$,
\begin{equation}\label{eq 4.17}
\begin{split}
D_1&\lesssim  \left(\int_0^r\int_{\mathbb{R}^2}(r-s)^{-2\eta}
|D(r-s,y,h)|^2\cdot \|\Delta_yv(s,\cdot)\|_{L_{\lambda}^p(\Omega\times\mathbb{R})}^2|h|^{2H-2}dhdyds\right)^{\frac{p}{2}} \\
&\hskip 12pt+\left(\int_0^r\int_{\mathbb{R}^2}(r-s)^{-2\eta}|D(r-s,y,h)|^2\cdot \|v(s,\cdot)\|_{L_{\lambda}^p(\Omega\times\mathbb{R})}^2|h|^{2H-2}dhdyds\right)^{\frac{p}{2}} \\
&\lesssim \left(\int_0^r\int_{\mathbb{R}}(r-s)^{-2\eta-\frac{1}{\alpha}}\|\Delta_yv(s,\cdot)\|_{L_{\lambda}^p(\Omega\times\mathbb{R})}^2|y|^{2H-2}dyds\right)^{\frac{p}{2}} \\
&\hskip 12pt+\left(\int_0^r(r-s)^{-2\eta-\frac{2-2H}{\alpha}} \|v(s,\cdot)\|_{L_{\lambda}^p(\Omega\times\mathbb{R})}^2 ds\right)^{\frac{p}{2}}.
\end{split}
\end{equation}

For the second term $D_2$,   by Minkowski's inequality, we have
\begin{align}\label{eq 4.18}
D_2&\lesssim\int_0^r\int_{\mathbb{R}^2}(r-s)^{-2\eta}
|G_{\alpha}(r-s,y)|^2\left( \int_{\mathbb{R}}\|\Delta_hv(s,y+z)\|_{L^p(\Omega)}^p\lambda(z)dz \right)^{\frac{2}{p}}|h|^{2H-2}dhdyds\nonumber\\
&\lesssim  \int_0^r\int_{\mathbb{R}}(r-s)^{-2\eta-\frac{1}{\alpha}}\left(\int_{\mathbb{R}}\int_{\mathbb{R}}(r-s)^{\frac{1}{\alpha}}
G_{\alpha}(r-s,y)^2\|\Delta_hv(s,z)\|_{L^p(\Omega)}^pdz\lambda(z-y)dy\right)^{\frac{2}{p}}|h|^{2H-2}dhds\nonumber\\
&\lesssim  \int_0^r\int_{\mathbb{R}}(r-s)^{-2\eta-\frac{1}{\alpha}}\|\Delta_hv(s,\cdot)\|_{L^p_{\lambda}(\Omega\times\mathbb{R})}^2 |h|^{2H-2}dhds,
\end{align}
where in the second inequality, we use Jensen's inequality  with respect to the probability measure 
$$\mu(\cdot)=\frac{1}{c_{\alpha}}\int_{\cdot}(r-s)^{\frac{1}{\alpha}}G_{\alpha}^2(r-s,y) dy,$$
 with
  $$c_{\alpha}=\int_{\mathbb{R}}(r-s)^{\frac{1}{\alpha}}G_{\alpha}(r-s,y)^2dy\simeq \int_{\mathbb{R}} G_{\alpha}(1,z)^2dz<\infty,$$ 
  and the function $\phi(x)=x^{2/p}, x>0,$ is concave  when $p>2$. The last inequality follows from  \eqref{eq power esti3}.

Recall the norm  $\|v\|_{\mathcal{Z}_{\lambda,T}^p}$ defined in \eqref{eq norm Lam}.   The estimates   \eqref{eq 4.17} and \eqref{eq 4.18} imply
\begin{equation}\label{eq 4.19}
 \mathbb{E}\left\|J(\eta,r,\cdot)\right\|_{L^p_{\lambda}(\mathbb{R})}^p\lesssim \|v\|_{\mathcal{Z}_{\lambda,T}^p}^p\left(\int_0^r(r-s)^{-2\eta-\frac{2-2H}{\alpha}} +(r-s)^{-2\eta-\frac{1}{\alpha}}ds\right)^{\frac{p}2}.
\end{equation}
If we have $-2\eta-\frac{2-2H}{\alpha}>-1$ and $-2\eta-\frac{1}{\alpha}>-1$, i.e.,
\begin{align}\label{eq cond eta}
\eta<\frac{2H+\alpha-2}{2\alpha},
\end{align} 
then \eqref{eq 4.16} follows.
 However,  condition \eqref{eq cond eta} should be combined with \eqref{eq p condi2}. This yields $$\frac{\alpha+1}{\alpha p}<\eta<\frac{2H+\alpha-2}{2\alpha},$$ which is possible only if  $p>\frac{2(\alpha+1)}{2H+\alpha-2}$. Thus, under the condition of the proposition, the inequality \eqref{eq 4.16} holds. This completes the proof of (i).

\noindent{\bf Step 3.} In this and next step we prove Part (ii). The spirit of the proof is similar to that of the proof of (i) but is more
involved. To obtain the desired decay rate of $\mathcal{N}_{\frac12-H}\Phi(t,x)$, we again  use the representation \eqref{eq 4.13} to express $\Phi(t,x)$ in terms of  $J$:
\begin{equation*}
\begin{split}
\Phi(t,x+h)-\Phi(t,x)&=\frac{\sin(\pi \eta)}{\pi}\int_0^{t}\int_{\mathbb{R}}(t-r)^{\eta-1}D(t-r,x-z,h)J(\eta,r,z)dzdr\\
&=\frac{\sin(\pi \eta)}{\pi}\int_0^{t}\int_{\mathbb{R}}(t-r)^{\eta-1}G_{\alpha}(t-r,x-z)\Delta_{h}J(\eta,r,z)dzdr,
\end{split}\end{equation*}
where  $\Delta_{h}J(\eta, t,x):=J(\eta,t,x+h)-J(\eta,t,x)$.

By Minkowski's inequality and  H\"older's inequality, we have
\begin{align*}
&\int_{\mathbb{R}} \left|\Phi(t,x+h)-\Phi(t,x)\right|^2|h|^{2H-2}dh\\
 \simeq &\, \int_{\mathbb{R}}\left|\int_0^{t}\int_{\mathbb{R}}(t-r)^{\eta-1}G_{\alpha}(t-r,x-z)\Delta_{h}J(\eta,r,z)dzdr\right|^2|h|^{2H-2}dh\\
 \lesssim &\,\left( \int_0^t\int_{\mathbb{R}}(t-r)^{\eta-1}G_{\alpha}(t-r,x-z)\left[\int_{\mathbb{R}}|\Delta_{h}J(\eta,r,z)|^2|h|^{2H-2}dh\right]^{\frac12}  dzdr\right)^2\\
 \lesssim &\,\left( \int_0^t\int_{\mathbb{R}}(t-r)^{q(\eta-1)}G_{\alpha}(t-r,x-z)^q\lambda(z)^{-\frac{q}{p}}dzdr\right)^{\frac2q}\\
&\hskip 12pt\cdot \left( \int_0^t\int_{\mathbb{R}}\left[\int_{\mathbb{R}}|\Delta_{h}J(\eta,r,z)|^2|h|^{2H-2}dh\right]^{\frac{p}2}\lambda(z)dzdr \right)^{\frac2p}\\
 \lesssim &\, \lambda(x)^{-\frac2p}\left(\int_0^t(t-r)^{q(\eta-1)-\frac{q-1}{\alpha}}dr\right)^{\frac2q} \left( \int_0^t\int_{\mathbb{R}}\left[\int_{\mathbb{R}}|\Delta_{h}J(\eta,r,z)|^2|h|^{2H-2}dh\right]^{\frac{p}2}\lambda(z)dzdr \right)^{\frac2p},
\end{align*}
where    \eqref{eq power esti2} is used in the last inequality provided that 
$$\frac{2q(1-H)}{p}<(1+\alpha)q-1,$$
that is, 
 \begin{equation}
    \label{eq p condi3}p>\frac{1-2H}{\alpha}.
    \end{equation}

Take $\theta=\frac1p$  and assume   
   \begin{equation}\label{eq p condi4}
   \eta>\frac{\alpha+1}{p\alpha}.
   \end{equation}
Note that if  $p>\frac{\alpha+1}{\alpha}$,  then \eqref{eq p condi3} follows immediately since $\alpha+2H>2$.  Consequently, 
\begin{align*}
& \sup_{{t\in[0,T],\,x\in\mathbb R}} \lambda(x)^{\theta}\left(\int_{\mathbb{R}}\left|\Phi(t,x+h)-\Phi(t,x)\right|^2|h|^{2H-2}dh\right)^{\frac12}\\
 \lesssim &\,  \left( \int_0^t\int_{\mathbb{R}}\left[\int_{\mathbb{R}}|\Delta_{h}J(\eta,r,z)|^2|h|^{2H-2}dh\right]^{\frac{p}2}\lambda(z)dzdr \right)^{\frac1p}.
\end{align*}
 Thus, to prove part (ii), it suffices to show that   there exists some constant $c>0$, independent of $r\in[0,T]$, such that
 \begin{equation}\label{eq 4.21}
 I:=\mathbb{E}\int_{\mathbb{R}}\left[\int_{\mathbb{R}}|\Delta_{h}J(\eta,r,z)|^2|h|^{2H-2}dh\right]^{\frac{p}2}\lambda(z)dz\le c\|v\|_{\mathcal{Z}_{\lambda,T}^p}^p.
 \end{equation}
 
\noindent{\bf Step 4.} In this step we prove   inequality \eqref{eq 4.21}. Recall  $J$ defined in   \eqref{eq 4.12}.   By Minkowski's inequality and    Burkholder-Davis-Gundy's inequality \eqref{eq BDG},  we have
\begin{align*}
I&\lesssim \left(\int_{\mathbb{R}}\left[\int_{\mathbb{R}}\mathbb{E}|\Delta_{h}J(\eta,r,z)|^p\lambda(z)dz\right]^{\frac2p}|h|^{2H-2}dh\right)^{\frac{p}2}\\
&\lesssim \Bigg(\int_{\mathbb{R}}\Bigg[\int_{\mathbb{R}}\mathbb{E}\Bigg(\int_0^r\int_{\mathbb{R}^2}(r-s)^{-2\eta}
\big|D(r-s,z-y-l,h)v(s,y+l)\\
&\hskip 12pt\,\, -D(r-s,z-y,h)v(s,y)\big|^2|l|^{2H-2}dldyds\Bigg)^{\frac{p}2}\lambda(z)dz\Bigg]^{\frac{2}{p}}|h|^{2H-2}dh\Bigg)^{\frac{p}2}.
\end{align*}
Define 
$$
\mathcal{I}_1(r,z,h):=\mathbb{E}\left(\int_0^r\int_{\mathbb{R}^2}(r-s)^{-2\eta}|D(r-s,z-y,h)|^2|\Delta_lv(s,y)|^2|l|^{2H-2}dldyds\right)^{\frac{p}{2}},
$$ 
$$
\mathcal{I}_2(r,z,h):=\mathbb{E}\left(\int_0^r\int_{\mathbb{R}^2}(r-s)^{-2\eta}|\Box(r-s,z-y,l,h)|^2| v(s,y)|^2|l|^{2H-2}dldyds\right)^{\frac{p}{2}}.
$$
By \eqref{nota alpha 2}, we have
\begin{align*}
&\mathbb{E}\int_{\mathbb{R}} \left[\int_{\mathbb{R}}|\Delta_{h}J(\eta,r,z)|^2|h|^{2H-2}dh\right]^{\frac{p}2}\lambda(z)dz\\
 \lesssim&\, \left(\int_{\mathbb{R}}\left[\int_{\mathbb{R}}\mathcal{I}_1(r,z,h)\lambda(z)dz\right]^{\frac{2}{p}}|h|^{2H-2}dh\right)^{\frac{p}2}+\left(\int_{\mathbb{R}}\left[\int_{\mathbb{R}}\mathcal{I}_2(r,z,h)\lambda(z)dz\right]^{\frac{2}{p}}|h|^{2H-2}dh\right)^{\frac{p}2}\\
=:&\, I_1^{\frac{p}2}+I_2^{\frac{p}2}.
\end{align*}
Using  the  change of variables $y\to z-y$ and   Minkowski's inequality, we have 
\begin{equation}\label{eq 4.22}
\begin{split}
I_1&\lesssim \int_{\mathbb R} \Bigg|\mathbb{E}\Bigg(\int_0^r\int_{\mathbb{R}^2}(r-s)^{-2\eta}|D(r-s,y,h)|^2 \\
&\hskip12pt \ \ \cdot |\Delta_lv(s,y+z)|^2|l|^{2H-2}dldyds\Bigg)^{\frac{p}{2}}\lambda(z)dz\Bigg|^{\frac{2}{p}}|h|^{2H-2}dh \\
&\lesssim \int_0^r\int_{\mathbb{R}^3}(r-s)^{-2\eta}|D(r-s,y,h)|^2|l|^{2H-2}|h|^{2H-2} \\
&\hskip12pt \,\,  \cdot\left(\int_{\mathbb{R}}\mathbb{E}|\Delta_lv(s,z)|^p\lambda(z-y)dz\right)^{\frac2p}dydhdlds.
\end{split}
\end{equation}
 Since $x^{2/p}$, $x>0$,  is   concave   for $p\ge2$, we  may  apply Jensen's inequality with respect to the probability measure
$$
\frac{1}{c_{\alpha}}(r-s)^{\frac{2-2H}{\alpha}} [G_{\alpha}(r-s,y+h)-G_{\alpha}(r-s,y)]^2|h|^{2H-2}dydh,
$$ 
with the  normalization constant
$$c_{\alpha}=(r-s)^{\frac{2-2H}{\alpha}}\int_{\mathbb{R}^2}|D(r-s,x,h)|^2|h|^{2H-2}dhdx, $$
which is finite  by \eqref{eq  2.17} with $\beta=\frac12-H$. 
 
 Thus, 
by the first inequality in Lemma \ref{lem 2.12},   for $p\ge2$, 
\begin{equation}\label{eq 4.23}
\begin{split}
I_1 \lesssim&\, \int_0^r\int_{\mathbb{R}}(r-s)^{-2\eta-\frac{2-2H}{\alpha}}\Bigg(\int_{\mathbb{R}^2}(r-s)^{\frac{2-2H}{\alpha}}|D(r-s,y,h)|^2 \\
&\hskip12pt \hskip12pt |h|^{2H-2}\int_{\mathbb{R}}\mathbb{E}|\Delta_lv(s,z)|^p\lambda(z-y)dzdydh\Bigg)^{\frac{2}{p}}|l|^{2H-2}dlds \\
 \lesssim&\, \int_0^r\int_{\mathbb{R}}(r-s)^{-2\eta-\frac{2-2H}{\alpha}}\|\Delta_lv(s,\cdot)\|^2_{L^p_{\lambda}(\Omega\times \mathbb{R})}|l|^{2H-2}dlds.
\end{split}
\end{equation}
  Meanwhile,
\begin{equation}\label{eq 4.24}
\begin{split}
I_2(r,z,h) \lesssim & \,  \mathbb{E}\left(\int_0^r\int_{\mathbb{R}^2}(r-s)^{-2\eta}|\Box(r-s,y,l,h)|^2| v(s,z)|^2|l|^{2H-2}dldyds\right)^{\frac{p}{2}} \\
&+\mathbb{E}\left(\int_0^r\int_{\mathbb{R}^2}(r-s)^{-2\eta}|\Box(r-s,y,l,h)|^2| \Delta_yv(s,z)|^2|l|^{2H-2}dldyds\right)^{\frac{p}{2}}\\
 :=&\,\mathcal  I_{21}(r,z,h)+\mathcal I_{22}(r,z,h) \\
\end{split}
\end{equation}
Using  Minkowski's inequality,   Lemma \ref{lem 2.8}, and Lemma \ref{lem 2.12}, we have
\begin{equation}\label{eq 4.25}
\begin{split}
I_{21}&:=\int_{\mathbb{R}}\left[\int_{\mathbb{R}}\mathcal{I}_{21}(r,z,h)\lambda(z)dz\right]^{\frac2{p}}|h|^{2H-2}dh \\
&\lesssim \int_0^r (r-s)^{-2\eta+\frac{4H-3}{\alpha}} \| v(s,\cdot)\|^2_{L^p_{\lambda}(\Omega\times \mathbb{R})}ds,
\end{split}
\end{equation}
and
\begin{equation}\label{eq 4.26}
\begin{split}
I_{22}&:=\int_{\mathbb{R}}\left[\int_{\mathbb{R}}\mathcal{I}_{22}(r,z,h)\lambda(z)dz\right]^{\frac2{p}}|h|^{2H-2}dh \\
&\lesssim \int_0^r (r-s)^{-2\eta+\frac{2H-2}{\alpha}} \| \Delta_yv(s,\cdot)\|^2_{L^p_{\lambda}(\Omega\times \mathbb{R})}|y|^{2H-2}dy ds.
\end{split}
\end{equation}
Recalling the definition of $\|v\|_{\mathcal{Z}_{\lambda,T}^p}$ and combining \eqref{eq 4.23}, \eqref{eq 4.25}, and \eqref{eq 4.26}, we obtain
\begin{equation}\label{eq 4.27}
\begin{split}
&\mathbb{E}\int_{\mathbb{R}} \left[\int_{\mathbb{R}}|\Delta_{h}J(\eta,r,z)|^2|h|^{2H-2}dh\right]^{\frac{p}2}\lambda(z)dz \\
 \lesssim &\,  \|v\|_{\mathcal{Z}_{\lambda,T}^p}^p\left( \int_0^r (r-s)^{-2\eta+\frac{4H-3}{\alpha}} +(r-s)^{-2\eta+\frac{2H-2}{\alpha}}ds\right)^{\frac{p}{2}}.
\end{split}
\end{equation}
To ensure the finiteness of the above integral,  it requires that 
 $$
 \eta<\frac{4H-3+\alpha}{2\alpha}.
 $$
  This explains   the assumption  $H> \frac{3-\alpha}{4}$. Combining this condition with \eqref{eq p condi4} yields
   $$
   \frac{\alpha+1}{p\alpha}<\eta<\frac{4H-3+\alpha}{2\alpha}.
   $$ 
  Therefore,   \eqref{eq 4.21} holds when 
  $$
  p>\frac{2\alpha+2}{4H-3+\alpha}.
  $$ 
  The proof of (ii) is now complete.

\noindent{\bf Step 5.} We now  prove Part (iii). We continue to use the representation \eqref{eq 4.13}. Without loss of generality, we   assume $h>0$ and
$ t \in [0, T ]$ such that $t + h \le T $. For $\eta\in(0,1)$, we have
\begin{align*}
\Phi (t+h,x)-\Phi(t,x) 
 =&\, \frac{\sin(\pi \eta)}{\pi}\Bigg[\int_0^{t+h}\int_{\mathbb{R}}(t+h-r)^{\eta-1}G_{\alpha}(t+h-r,x-z)J(\eta,r,z)dzdr\\
&\hskip12pt -\int_0^{t}\int_{\mathbb{R}}(t-r)^{\eta-1}G_{\alpha}(t-r,x-z)J(\eta,r,z)dzdr\Bigg]\\
 \lesssim&\,  \sum_{i=1}^3\mathcal{J}_{i}(t,h,x),
\end{align*}
where
\begin{align*}
\mathcal{J}_1(t,h,x):=&\, \int_0^{t}\int_{\mathbb{R}}\left[(t+h-r)^{\eta-1}-(t-r)^{\eta-1}\right]G_{\alpha}(t-r,x-z)J(\eta,r,z)dzdr,   \\
\mathcal{J}_2(t,h,x):=&\, \int_0^{t}\int_{\mathbb{R}}(t+h-r)^{\eta-1}\left[G_{\alpha}(t+h-r,x-z)-G_{\alpha}(t-r,x-z)\right]J(\eta,r,z)dzdr,\\
  \mathcal{J}_3(t,h,x):= &\, \int_{t}^{t+h}\int_{\mathbb{R}}(t+h-r)^{\eta-1} G_{\alpha}(t+h-r,x-z) J(\eta,r,z)dzdr.
 \end{align*}
As in the proofs of parts (i) and (ii), we insert additional factors of $\lambda^{-\frac1p}(z)\cdot \lambda^{\frac1p}(z)$ and apply H\"older's inequality together with  \eqref{eq power esti2}  to estimate  $\mathcal{J}_1$.  

  For $p>\frac{1-2H}{\alpha}$, we have
    \begin{equation}\label{eq 4.28}
    \begin{split}
\mathcal{J}_1(t,h,x)\le &\,  \lambda^{-\frac1p}(x)\left(\int_0^t\left|(t+h-r)^{\eta-1}-(t-r)^{\eta-1}\right|^q(t-r)^{-\frac{q-1}{\alpha}}{dr}\right)^{\frac1q}\\
&\, \cdot
\left(\int_0^T\left\|J(\eta,r,\cdot)\right\|^p_{L_\lambda^p(\mathbb{R})}dr\right)^{\frac1p}.
\end{split}
\end{equation}

Fix $\gamma\in(0,1)$. By the elementary inequality $$\left|(t+h-r)^{\eta-1}-(t-r)^{\eta-1}\right|\lesssim |t-r|^{\eta-1-\gamma}h^{\gamma},$$
(see \cite[Page 264]{SS2002} or \cite[(4.29)]{HW2022}),   we have
\begin{equation}\label{eq cond ph}
\begin{split}
& \sup_{{t\in[0,T],\,x\in\mathbb R}}\lambda(x)^{\frac1p}\left|\mathcal{J}_1(t,h,x)\right|\\
\lesssim &\, h^{\gamma}\sup_{t\in[0,T]}\left( \int_0^t(t-r)^{q(\eta-1-\gamma)-\frac{q-1}{\alpha}}dr\right)^{\frac1q}\left(\int_0^T\left\|J(\eta,r,\cdot)\right\|^p_{L_\lambda^p(\mathbb{R})}dr\right)^{\frac1p}.
\end{split}
\end{equation}
When  $$\frac{\alpha+1}{\alpha p}+\gamma<\eta<\frac{2H+\alpha-2}{2\alpha},$$ 
which is possible provided that 
$$\gamma<\frac{2H+\alpha-2}{2\alpha}-\frac{\alpha+1}{\alpha p},$$
 by   \eqref{eq 4.16} and \eqref{eq cond ph},  we have
\begin{equation}\label{eq 4.30}
\mathbb{E}\left|\sup_{{t\in[0,T],\,x\in\mathbb R}} \lambda(x)^{\frac1p}\mathcal{J}_1(t,h,x)\right|^p\lesssim |h|^{p\gamma}\|v\|_{\mathcal{Z}_{\lambda,T}^p}^p.
\end{equation}

We now turn to bounding $\mathcal{J}_2(t, h, x)$. By  H\"older's  inequality,
\begin{align*}
\mathcal{J}_2(t, h, x)  \lesssim &\, \Bigg(\int_0^t\int_{\mathbb{R}} (t+h-r)^{q(\eta-1)}\big|G_{\alpha}(t+h-r,x-z)-G_{\alpha}(t-r,x-z)\big|^q \lambda^{-\frac{q}{p}}(z)dzdr\Bigg)^{\frac1q} \\
&\ \ \  \ \cdot \left(\int_0^T\left\|J(\eta,r,\cdot)\right\|^p_{L_\lambda^p(\mathbb{R})}dr\right)^{\frac1p}.
\end{align*}

For any $\gamma\in(0,1)$, we have
\begin{align*}
&\big|G_{\alpha}(t+h-r,x-z)-G_{\alpha}(t-r,x-z)\big|\\
&\le\big|G_{\alpha}(t+h-r,x-z)+G_{\alpha}(t-r,x-z)\big|^{1-\gamma}\cdot \big|G_{\alpha}(t+h-r,x-z)-G_{\alpha}(t-r,x-z)\big|^{\gamma}\\
&\lesssim (t-r)^{-\gamma}h^{\gamma}\big|G_{\alpha}(t+h-r,x-z)+G_{\alpha}(t-r,x-z)\big|^{1-\gamma}G_{\alpha}(t-r,x-z)^{\gamma}\\
&\lesssim  (t-r)^{-\gamma}h^{\gamma}\left(G_{\alpha}(t+h-r,x-z)+G_{\alpha}(t-r,x-z)\right),
\end{align*}
where we use  \eqref{eq grad temp} in the second inequality. Consequently, 
\begin{align*}
\mathcal{J}_2(t, h, x)&\lesssim \Bigg(\int_0^t\int_{\mathbb{R}} (t+h-r)^{q(\eta-1)} \cdot(t-r)^{-q\gamma}h^{q\gamma}\big(G_{\alpha}(t+h-r,x-z)^{q}+G_{\alpha}(t-r,x-z)^{q}\big)\nonumber\\
&\hskip12pt\cdot\lambda^{-\frac{q}{p}}(z)dzdr\Bigg)^{\frac1q}\left(\int_0^T\left\|J(\eta,r,\cdot)\right\|^p_{L_\lambda^p(\mathbb{R})}dr\right)^{\frac1p}\nonumber\\
&\lesssim h^{\gamma}\lambda^{-\frac1p}(x)\Bigg(\int_0^t(t+h-r)^{q(\eta-1)} (t-r)^{-q\gamma} \cdot\left[(t+h-r)^{-\frac{q-1}{\alpha}}+(t-r)^{-\frac{q-1}{\alpha}}\right]dr\Bigg)^{\frac1q}\\
&\hskip24pt\cdot\left(\int_0^T\left\|J(\eta,r,\cdot)\right\|^p_{L_\lambda^p(\mathbb{R})}dr\right)^{\frac1p}\\
&\lesssim h^{\gamma}\lambda^{-\frac1p}(x)\Bigg(\int_0^t(t-r)^{q(\eta-1-\gamma)-\frac{q-1}{\alpha}}  dr\Bigg)^{\frac1q}\cdot\left(\int_0^T\left\|J(\eta,r,\cdot)\right\|^p_{L_\lambda^p(\mathbb{R})}dr\right)^{\frac1p}.
\end{align*}
In the second inequality above,   \eqref{eq power esti2}  is used, which is valid for  $p>\frac{1-2H}{\alpha}$; the last inequality  uses the fact that   $(t+h-r)^{\theta}\le (t-r)^{\theta}$ for $\theta<0$.

  When  $\gamma<\frac{2H+\alpha-2}{2\alpha}-\frac{\alpha+1}{\alpha p}$, by \eqref{eq 4.16},  we have
\begin{equation}\label{eq 4.33}
\mathbb{E}\left|\sup_{t\in [0,T],\, x\in \mathbb R}\lambda(x)^{\frac1p}\mathcal{J}_2(t,h,x)\right|^p\lesssim |h|^{p\gamma}\|v\|_{\mathcal{Z}_{\lambda,T}^p}^p.
\end{equation}

Similarly to the proof of \eqref{eq 4.28}, we have that for $\gamma<\frac{2H+\alpha-2}{2\alpha}-\frac{\alpha+1}{\alpha p},$
\begin{equation}\label{eq 4.34}
\mathbb{E}\left|\sup_{t\in [0,T],\, x\in \mathbb R}\lambda(x)^{\frac1p}\mathcal{J}_3(t,h,x)\right|^p\lesssim |h|^{p\gamma}\|v\|_{\mathcal{Z}_{\lambda,T}^p}^p.
\end{equation}
Combining \eqref{eq 4.30}, \eqref{eq 4.33}, and \eqref{eq 4.34} yields    \eqref{eq 4.10}. This completes the proof of part (iii).

{\noindent \bf{Step 6.}} We now prove Part (iv). Using  \eqref{eq 4.13} and    H\"older's inequality, we have
\begin{equation}\label{eq J3-1}
\begin{split}
& \left|\Phi(t,x)-\Phi(t,y)\right|\\
 =&\, \left|\frac{\sin(\pi \eta)}{\pi}\int_0^{t}\int_{\mathbb{R}}(t-r)^{\eta-1}[G_{\alpha}(t-r,x-z)-G_{\alpha}(t-r,y-z)]J(\eta,r,z)dzdr\right| \\
 \lesssim&\, \left(\int_0^t\int_{\mathbb{R}}(t-r)^{q(\eta-1)}\left|G_{\alpha}(t-r,x-z)-G_{\alpha}(t-r,y-z)\right|^{q}\lambda^{-\frac{q}{p}}(z)dzdr\right)^{\frac1q} \\
&\hskip10pt \cdot \left( \int_0^t\int_{\mathbb{R}} \left|J(\eta,r,z)\right|^p \lambda(z)dzdr\right)^{\frac1p}.
\end{split}
\end{equation}
 For any $\gamma\in(0,1)$,  by \eqref{eq grad spat},   we have
\begin{align*}
&\big|G_{\alpha}(t-r,x-z)-G_{\alpha}(t-r,y-z)\big| \\
 \le &\, \big(G_{\alpha}(t-r,x-z)+G_{\alpha}(t-r,y-z)\big)^{1-\gamma}\cdot\big|G_{\alpha}(t-r,x-z)-G_{\alpha}(t-r,y-z)\big|^{\gamma} \\
 \lesssim \, &  (t-r)^{-\frac{\gamma}{\alpha}}|x-y|^{\gamma}\big(G_{\alpha}(t-r,x-z)+G_{\alpha}(t-r,y-z)\big).
\end{align*}
Substituting this  into \eqref{eq J3-1} and using \eqref{eq power esti2} for $p>\frac{1-2H}{\alpha}$, we have
\begin{equation*}
\begin{split}
&\left|\Phi(t,x)-\Phi(t,y)\right|\\
\lesssim&\,  |x-y|^{\gamma} \Bigg(\int_0^t\int_{\mathbb{R}}(t-r)^{q(\eta-1)}(t-r)^{-\frac{q\gamma}{\alpha}}\Bigg(G_{\alpha}(t-r,x-z)^q\\
&+G_{\alpha}(t-r,y-z)^q\Bigg)\lambda^{-\frac{q}{p}}(z)dzdr\Bigg)^{\frac1q} \cdot\left(\int_0^T\left\|J(\eta,r,\cdot)\right\|^p_{L_\lambda^p(\mathbb{R})}dr\right)^{\frac1p}\\
 \lesssim&\,  |x-y|^{\gamma}\left(\lambda^{-\frac{1}{p}}(x)+\lambda^{-\frac{1}{p}}(y)\right)\left(\int_0^t(t-r)^{q(\eta-1)-\frac{q\gamma}{\alpha}-\frac{q-1}{\alpha}}dr\right)^{\frac1q}\\
&\hskip12pt\cdot\left(\int_0^T\left\|J(\eta,r,\cdot)\right\|^p_{L_\lambda^p(\mathbb{R})}dr\right)^{\frac1p}.
\end{split}
\end{equation*}
If $$q(\eta-1)-\frac{q\gamma}{\alpha}-\frac{q-1}{\alpha}>-1, \ \  \eta<\frac{2H+\alpha-2}{2\alpha},$$ i.e.,
  $$\frac{\alpha+1}{\alpha p}+\frac{\gamma}{\alpha}<\eta<\frac{2H+\alpha-2}{2\alpha},$$ then,   by \eqref{eq 4.16}, we have 
\begin{align*}
\mathbb{E}\left|\sup_{t\in[0,T],\,x\in\mathbb{R}}\frac{\Phi(t,x)-\Phi(t,y)}{\lambda^{-\frac1p}(x)+\lambda^{-\frac1p}(y)}\right|^p\lesssim |x-y|^{p\gamma}\|v\|_{\mathcal{Z}_{\lambda,T}^p}^p.
\end{align*}
This proves \eqref{eq 4.11}, and the proof   is  complete.
\end{proof}

\section{Existence and uniqueness  of the solution}\label{sec solution}

First,  we give the definitions of strong (mild) and weak solutions.
\begin{definition}\label{def mild}
  \begin{itemize}
\item[(i)] A    real-valued adapted stochastic process $u(t,x)$ is called a  strong  (mild) solution  to \eqref{SFHE}, if for all $t\ge0$ and $x\in\mathbb{R}$,
\begin{equation}\label{eq solution}
u(t,x)=G_{\alpha}(t,\cdot)*u_0(x)+\int_0^t\int_{\mathbb{R}}G_{\alpha}(t-s,y-x)\sigma(s,y,u(s,y))W(ds,dy)\ \ \text{a.s.},
\end{equation}
  where the stochastic integral is understood in the sense of Proposition \ref{prop 2.3}.
\item[(ii)] We say \eqref{SFHE} has a {\it weak solution}, if there exists a probability space with a filtration $(\widetilde{\Omega},\widetilde{\mathcal{F}},\widetilde{\mathbb{P}},\widetilde{\mathcal{F}_t})$, a Gaussian random field $\widetilde{W}$ identical to $W$ in law, and an adapted stochastic process  $\{u(t,x),t\ge0,x\in\mathbb{R}\}$ on this probability space   such that $u(t,x)$ is a mild solution to \eqref{SFHE} with respect to $(\widetilde{\Omega},\widetilde{\mathcal{F}},\widetilde{\mathbb{P}},\widetilde{\mathcal{F}_t})$ and $\widetilde{W}$.
\end{itemize}
 \end{definition}

Next, we establish the  existence and uniqueness of a solution   in $\mathcal{C}([0,T]\times\mathbb{R})$, the space of all continuous real-valued functions on $[0,T]\times\mathbb{R}$, equipped with the metric
\begin{equation}\label{distance}
d_{\mathcal{C}}(u,v):=\sum_{n=1}^{\infty}\frac{1}{2^n}\max_{0\le t\le T,|x|\le n} \left(|u(t,x)-v(t,x)|\wedge 1\right).
\end{equation}

Recall that the space $\mathcal{Z}_{\lambda,T}^p$ consists of random fields $v(t,x)$ such that  the norm $\|v\|_{\mathcal{Z}_{\lambda,T}^p}$ defined in \eqref{eq norm Lam} is finite. 
We will show that the solution to \eqref{SFHE} lies in $\mathcal{Z}_{\lambda,T}^p$ via approximation.

 \subsection{The approximate solution}

Following  \cite[Section 4.3]{HW2022}, we approximate the noise $W$  by the following smoothing  procedure. 

For any $\varepsilon>0$,    define 
\begin{equation}\label{eq 4.38}
\frac{\partial}{\partial x}W_{\varepsilon}(t,x):=\int_{\mathbb{R}}\rho_{\varepsilon}(x-y)W(t,dy),
\end{equation}
where 
$$\rho_{\varepsilon}(x):=(2\pi\varepsilon)^{-\frac12}e^{-\frac{x^2}{2\varepsilon}}.$$ 
The noise $W_{\varepsilon}$ induces an approximation of the mild solution:
\begin{equation}\label{eq 4.39}
u_{\varepsilon}(t,x)=G_{\alpha}(t,\cdot)* u_0(x)+\int_0^t\int_{\mathbb{R}}G_{\alpha}(t-s,x-y)\sigma(s,y,u_{\varepsilon}(s,y))W_{\varepsilon}(ds,dy),
\end{equation}
where the stochastic  integral is understood in the It\^{o} sense. As in \cite{HHLNT2017,HW2022}, thanks to the spatial regularity, the existence and uniqueness of the solution $u_{\varepsilon} $ to    \eqref{eq 4.39} is well-known via  Picard iteration.

 The following lemma states that the approximate solution $u_{\varepsilon}$ is uniformly bounded in space $\mathcal{Z}_{\lambda,T}^p$.  
 \begin{lemma}\label{lem 4.5} 
 Let $\frac{3-\alpha}{4}<H<\frac12$. Assume that  $\sigma(t,x,u)$ satisfies  (H1), and that  the initial value $u_0(x)\in\mathcal{Z}_{\lambda,0}^p$. Then  the approximate solution $u_{\varepsilon}$ satisfies
 \begin{equation}\label{eq 4.40}
 \sup_{\varepsilon>0}\|u_{\varepsilon}\|_{\mathcal{Z}_{\lambda,T}^p} :=\sup_{\varepsilon>0}\sup_{t\in[0,T]}\|u_{\varepsilon}(t,\cdot)\|_{L^p(\Omega\times \mathbb{R})}+\sup_{\varepsilon>0}\sup_{t\in[0,T]}\mathcal{N}^*_{\frac12-H,p}u_{\varepsilon}(t)<\infty.
 \end{equation}
  \end{lemma}
  
  Before proving Lemma \ref{lem 4.5}, we first  state the following   result, which shows that  the space  $\mathcal{Z}_{\lambda,T}^p$ is closed under convergence.
\begin{lemma}\cite[Lemma 4.6]{HW2022} \label{lem 4.6}
Assume that the  random fields $\{u_{\varepsilon}(t,x), {t\in[0,T],\,x\in\mathbb R}\}_{\varepsilon>0}\subset \mathcal{Z}_{\lambda,T}^p$ with
 $$
\sup_{\varepsilon>0}\|u_{\varepsilon}\|_{\mathcal{Z}_{\lambda,T}^p}=\sup_{\varepsilon>0}\left(\sup_{t\in[0,T]}\left\|u_{\varepsilon}(t,\cdot)\right\|_{L_{\lambda}^p(\Omega\times\mathbb{R})}
+\sup_{t\in[0,T]}\mathcal{N}_{\frac12-H,p}^*u_{\varepsilon}(t)\right)<\infty.
  $$
  If $u_{\varepsilon}\to u$ almost surely in $\left(\mathcal{C}([0,T]\times\mathbb{R}),d_{\mathcal{C}}\right)$ as $\varepsilon\to0$, then $u\in \mathcal{Z}_{\lambda,T}^p$.
\end{lemma}
\begin{proof}  The lemma is taken from \cite[Lemma 4.6]{HW2022}. Here, we provide a short alternative proof by a direct application of Fatou's lemma.
 Since $u_{\varepsilon}\to u$ in $\left(\mathcal{C}([0,T]\times\mathbb{R}),d_{\mathcal{C}}\right)$ almost surely, we have that for  each $t\in[0,T]$ and $x, h\in \mathbb{R}$,
$$u_{\varepsilon}(t,x)\to u(t,x), \text{  a.s. }$$
 and 
 $$\left|u_{\varepsilon}(t,x+h)-u_{\varepsilon}(t,x)\right|^2\to \left|u(t,x+h)-u(t,x)\right|^2, \text{  a.s. }$$   
  Thus, by Fatou's lemma,
$$
\left\|u(t,\cdot)\right\|_{L^p_{\lambda}(\Omega\times \mathbb{R})}\le \liminf\limits_{\varepsilon\to0}\left(\int_{\mathbb{R}}\mathbb{E}[|u_{\varepsilon}(t,x)|^p]\lambda(x)dx\right)^{\frac1{p}} <\infty,
$$
and
\begin{align*}\mathcal{N}_{\frac12-H,p}^*u(t)&=
\left(\int_{\mathbb{R}}\|u(t,\cdot+h)-u(t,\cdot)\|_{L_{\lambda}^p(\Omega\times \mathbb{R})}^2|h|^{2H-2}dh\right)^{\frac12}\\
&\le \liminf\limits_{\varepsilon\to0}\left(\int_{\mathbb{R}}\|u_{\varepsilon}(t,\cdot+h)-u_{\varepsilon}(t,\cdot)\|_{L_{\lambda}^p(\Omega\times \mathbb{R})}^2|h|^{2H-2}dh\right)^{\frac12}\\
&=\liminf\limits_{\varepsilon\to0}\mathcal{N}_{\frac12-H,p}^*u_{\varepsilon}(t)<\infty.
\end{align*}
  The proof is complete.
\end{proof}


\begin{proof}[Proof of Lemma \ref{lem 4.5}] 
We follow the   argument in \cite{HW2022}.  For notational simplicity   and  without loss of generality, we   assume $\sigma(t,x,u)=\sigma(u)$.  Define the Picard iteration as follows:
$$
u^0_{\varepsilon}(t,x):=G_{\alpha}(t,\cdot)*u_0(x),
$$
and recursively for $n\in \mathbb N$,
\begin{equation}\label{eq picard1}
u_{\varepsilon}^{n+1}(t, x):=G_{\alpha}(t,\cdot)*u_0(x)+\int_0^t\int_{\mathbb{R}}G_{\alpha}(t-s,x-y)\sigma\left(u_{\varepsilon}^n(s,y)\right)W_{\varepsilon}(ds,dy).
\end{equation}

As in \cite{HHLNT2017}, due to the spatial regularity, for any fixed $\varepsilon>0$, 
 the sequence $u_{\varepsilon}^n(t,x)$ converges to $u_{\varepsilon}(t,x)$ almost surely as  $n\to\infty$.  In  Steps 1 and 2 below, we first bound $\|u_{\varepsilon}^n(t,x)\|_{\mathcal{Z}_{\lambda,T}^p}$ uniformly in $n$ and $\varepsilon$. Then,   we use Fatou's lemma to establish \eqref{eq 4.40} in Step 3.
 
\noindent{\bf Step 1.}   
 Rewriting \eqref{eq picard1} gives
    $$
    u_{\varepsilon}^{n+1}(t,x)=G_{\alpha}(t,\cdot)*u_0(x)+\int_0^t\int_{\mathbb{R}}\left[\left(G_{\alpha}(t-s,x-\cdot)\sigma\left(u_{\varepsilon}^n(s,\cdot)\right)\right)*\rho_{\varepsilon}\right](y)W(ds,dy).
    $$
We  continue to use the notations $D(t,x,h)$ and $\Box(t-s,x,y,h)$ defined earlier  in \eqref{nota alpha  1} and \eqref{nota alpha 2}. 

Applying the Burkholder inequality,  the isometry property \eqref{eq iso}, and   the fact that  $|\sigma(u)|\le c( |u|+1)$, we have
\begin{equation}\label{eq 4.42}
\begin{split}
  \mathbb{E} \left[\left|u_{\varepsilon}^{n+1}(t,x)\right|^p\right]  
 \le &\,  c_p\mathbb{E}\left|G_{\alpha}(t,\cdot)*u_0(x)\right|^p+c_p\mathbb{E}\Bigg(\int_0^t\int_{\mathbb{R}^2}\big|G_{\alpha}(t-s,x-y-h)\sigma\left(u_{\varepsilon}^{n}(s,y+h)\right) \\
&\hskip 12pt     \,\,\, \, -G_{\alpha}(t-s,x-y)\sigma\left(u_{\varepsilon}^{n}(s,y)\right)\big|^2|h|^{2H-2}dhdyds\Bigg)^{\frac{p}2} \\
 \le &\,  c_p\left(\mathbb{E}\left|G_{\alpha}(t,\cdot)*u_0(x)\right|^p+\mathcal{D}_{1}^{\varepsilon,n}(t,x)+\mathcal{D}_{2}^{\varepsilon,n}(t,x)\right),
\end{split}
\end{equation}
where
\begin{align*}
\mathcal{D}_{1}^{\varepsilon,n}(t,x):=&\, \left(\int_0^t\int_{\mathbb{R}^2}\left|D(t-s,y,h)\right|^2\left(1+\|u_{\varepsilon}^n(s,x+y)
\|_{L^p(\Omega)}^2\right)|h|^{2H-2}dhdyds\right)^{\frac{p}2},\\
\mathcal{D}_{2}^{\varepsilon,n}(t,x):=&\, \left(\int_0^t\int_{\mathbb{R}^2}\left|G_{\alpha}(t-s,y)\right|^2\left\|\Delta_hu_{\varepsilon}^n(s,x+y)
\right\|_{L^p(\Omega)}^2|h|^{2H-2}dhdyds\right)^{\frac{p}2},
\end{align*}
with $\Delta_hu_{\varepsilon}^n(t,x):=u_{\varepsilon}^n(t,x+h)-u_{\varepsilon}^n(t,x)$.

By Jensen's inequality and \eqref{eq power esti1}, we have
\begin{align*}
\int_{\mathbb{R}}\mathbb{E}\left|G_{\alpha}(t,\cdot)*u_0(x)\right|^p\lambda(x)dx&\le \int_{\mathbb{R}}\int_{\mathbb{R}}\mathbb{E}|u_0(y)|^pG_{\alpha}(t,x-y)\lambda(x)dy dx\\
& \le c_{p,H,\alpha,T}\left\|u_0\right\|_{L^p_{\lambda}(\Omega\times\mathbb{R})}^p.
\end{align*}
 It follows that
  \begin{align}\label{eq 4.43}
\left\|u_{\varepsilon}^{n+1}(t,\cdot)\right\|^2_{L^p_{\lambda}(\Omega\times\mathbb{R})}
\le \, c_{p,H,\alpha,T}\left(\left\|u_0\right\|_{L^p_{\lambda}(\Omega\times\mathbb{R})}^2+ I_1^{\varepsilon,n}+I_2^{\varepsilon,n}\right),
\end{align}
where $I_1^{\varepsilon,n}$ and $I_2^{\varepsilon,n}$ are defined and bounded as follows:
\begin{equation}\label{eq 4.44}
\begin{split}
I_1^{\varepsilon,n}:=&\, \left(\int_{\mathbb{R}}\mathcal{D}_{1}^{\varepsilon,n}(t,x)\lambda(x)dx\right)^{\frac2{p}}\\
\le &\,  c_{p,H,\alpha}\int_0^t(t-s)^{\frac{2(H-1)}{\alpha}} \left(1+\|u_{\varepsilon}^n(s,\cdot)\|^2_{L^p_{\lambda}(\Omega\times\mathbb{R})}\right)ds,
\end{split}
\end{equation} and
\begin{equation}\label{eq 4.45}
\begin{split}
I_2^{\varepsilon,n}:=  \, \left(\int_{\mathbb{R}}\mathcal{D}_{2}^{\varepsilon,n}(t,x)\lambda(x)dx\right)^{\frac2{p}}  \le  \,  c_{p,H,\alpha}\int_0^t(t-s)^{-\frac{1}{\alpha}}\left[\mathcal{N}_{\frac12-H,p}^*u_{\varepsilon}^n(s)\right]^2ds.
\end{split}
\end{equation} 
These two estimates can be obtained by  arguments similar  to those in the proofs of \eqref{eq 4.17} and \eqref{eq 4.18}.

 Putting \eqref{eq 4.43}-\eqref{eq 4.45} together,  we have
\begin{equation}\label{eq 4.46}
\begin{split}
\left\|u_{\varepsilon}^{n+1}(t,\cdot)\right\|^2_{L^p_{\lambda}(\Omega\times\mathbb{R})}\le &\, c_{p,H,\alpha,T}\Bigg(\left\|u_0\right\|_{L^p_{\lambda}(\Omega\times\mathbb{R})}^2+\int_0^t(t-s)^{\frac{2(H-1)}{\alpha}} \left(1+\|u_{\varepsilon}^n(s,\cdot)\|^2_{L^p_{\lambda}(\Omega\times\mathbb{R})}\right)ds \\
&+\int_0^t(t-s)^{-\frac{1}{\alpha}}\left[\mathcal{N}_{\frac12-H,p}^*u_{\varepsilon}^n(s)\right]^2ds\Bigg).
\end{split}
\end{equation}

\noindent{\bf Step 2.}   Similarly to \eqref{eq 4.42}, we have
\begin{align*}
&\mathbb{E}\left[\left|u_{\varepsilon}^{n+1}(t,x)-u_{\varepsilon}^{n+1}(t,x+h)\right|^p\right]\\
 \le &\,  c_p\mathbb{E}\left[\left|G_{\alpha}(t,\cdot)*u_0(x)-G_{\alpha}(t,x+h)\right|^p\right]\\
&+ c_p\mathbb{E}\Bigg(\int_0^t\int_{\mathbb{R}^2}\Big|D(t-s,x-y-z,h)\sigma(u_{\varepsilon}^n(s,y+z)) 
  -D(t-s,x-z,h)\sigma(u_{\varepsilon}^n(s,z))\Big|^2|y|^{2H-2}dzdyds\Bigg)^{\frac{p}2}\\
 \le  &\, c_p\left[\mathbb{E}\left[\mathcal{I}_0(t,x,h)\right]+\mathbb{E}\left[\left(\mathcal{I}_1^{\varepsilon,n}(t, x, h)+\mathcal{I}_2^{\varepsilon,n}(t, x, h)\right)^{\frac{p}2}\right]\right],\end{align*}
where
\begin{align*}
&\mathcal{I}_0(t,x,h):=\left|G_{\alpha}(t,\cdot)*u_0(x)-G_{\alpha}(t,\cdot)*u_0(x+h)\right|^p,\\
&\mathcal{I}_1^{\varepsilon,n}(t,x,h):=\int_0^t\int_{\mathbb{R}^2}\left|D(t-s,x-y-z,h)\right|^2
\left|\sigma(u_{\varepsilon}^n(s,y+z))-\sigma(u_{\varepsilon}^n(s,z))\right|^2|y|^{2H-2}dzdyds,\\
&\mathcal{I}_2^{\varepsilon,n}(t,x,h):=\int_0^t\int_{\mathbb{R}^2}\left|\Box(t-s,x-z,y,h)\right|^2
\left|\sigma(u_{\varepsilon}^n(s,z))\right|^2|y|^{2H-2}dzdyds.
\end{align*}
By Minkowski's inequality, we obtain
\begin{equation}\label{eq 4.47}
\begin{split}
\left[\mathcal{N}_{\frac12-H,p}^*u_{\varepsilon}^{n+1}(t)\right]^2&=\int_{\mathbb{R}}\left|\int_{\mathbb{R}}\mathbb{E}\left[\left|
u_{\varepsilon}^{n+1}(t,x)-u_{\varepsilon}^{n+1}(t,x+h)\right|^p\right]\lambda(x)dx\right|^{\frac2p}|h|^{2H-2}dh \\
&\le  c_p\int_{\mathbb{R}}\left|\int_{\mathbb{R}}\mathbb{E}\mathcal{I}_0(t,x,h)\lambda(x)dx\right|^{\frac2p}|h|^{2H-2}dh \\
&\hskip8pt+c_p\sum_{i=1}^2\int_{\mathbb{R}}\left(\int_{\mathbb{R}}\mathbb{E}\left(\mathcal{I}_i^{\varepsilon,n}(t,x,h)^{\frac{p}2}\right)\lambda(x)dx\right)^{\frac2p}|h|^{2H-2}dh \\
&=:J_0+J_1+J_2.
\end{split}
\end{equation}

 The strategy for controlling  these three quantities is similar to that used  for the terms $\mathcal{I}_1 $ and  $\mathcal{I}_2 $ in Step 4 of the proof of Proposition \ref{prop 4.2}(ii).
 
 For the first term, we have
\begin{equation}\label{eq 4.48}
\begin{split}
J_0&\le c_p\int_{\mathbb{R}}\left(\int_{\mathbb{R}}\mathbb{E}\left|\int_{\mathbb{R}}G_{\alpha}(t,x-y)\Delta_hu_0(y)dy\right|^p\lambda(x)dx\right)^{\frac2p}|h|^{2H-2}dh \\
&\le c_p\int_{\mathbb{R}}\left(\int_{\mathbb{R}} \int_{\mathbb{R}}G_{\alpha}(t,x-y)\mathbb{E}|\Delta_hu_0(y)|^pdy \lambda(x)dx\right)^{\frac2p}|h|^{2H-2}dh  \\
&\le c_{p,H,\alpha,T}\int_{\mathbb{R}}\left(\int_{\mathbb{R}} \mathbb{E}|\Delta_hu_0(y)|^p \lambda(y)dy\right)^{\frac2p}|h|^{2H-2}dh  \\
&\le c_{p,H,\alpha,T}\left[\mathcal{N}_{\frac12-H,p}^*u_{0}\right]^2,
\end{split}
\end{equation}
where we used Jessen's inequality in the second inequality  and  estimate \eqref{eq power esti1}   in the   third.

Similarly to  the proofs of  \cite[(4.49)-(4.51)]{HW2022}, we have 
\begin{align*}
J_1 \le   &\, c_{p,H,\alpha,T}\int_{0}^t(t-s)^{\frac{2H-2}{\alpha}}\left[\mathcal{N}_{\frac12-H,p}^*u_{\varepsilon}^n(s)\right]^2ds,\\
J_2 \le &\, c_{p,H,\alpha,T}\int_{0}^t(t-s)^{\frac{4H-3}{\alpha}}\left(1+\|u_{\varepsilon}^n(s,\cdot)\|_{L_{\lambda}^p(\Omega\times\mathbb{R})}^2\right) +(t-s)^{\frac{2H-2}{\alpha}}\left[\mathcal{N}_{\frac12-H,p}^*u_{\varepsilon}^n(s)\right]^2ds. 
\end{align*}

Combining   \eqref{eq 4.47} with \eqref{eq 4.48}, we obtain
\begin{equation}\label{eq 4.52}
\begin{split}
\left[\mathcal{N}_{\frac12-H,p}^*u_{\varepsilon}^{n+1}(t)\right]^2 \le &\,  c_{p,H,\alpha,T}\Bigg(\left[\mathcal{N}_{\frac12-H,p}^*u_{0}\right]^2+\int_{0}^t(t-s)^{\frac{2H-2}{\alpha}}\left[\mathcal{N}_{\frac12-H,p}^*u_{\varepsilon}^n(s)\right]^2ds \\
  &+\int_{0}^t(t-s)^{\frac{4H-3}{\alpha}}\left(1+\|u_{\varepsilon}^n(s,\cdot)\|_{L_{\lambda}^p(\Omega\times\mathbb{R})}^2\right)\Bigg).
  \end{split}
\end{equation}

\noindent{\bf Step 3.} Define
$$
\Psi_{\varepsilon}^n(t):=\left\|u_{\varepsilon}^{n}(t,\cdot)\right\|^2_{L^p_{\lambda}(\Omega\times\mathbb{R})}+\left[\mathcal{N}_{\frac12-H,p}^*u_{\varepsilon}^{n}(t)\right]^2.
$$
By  \eqref{eq 4.46} and  \eqref{eq 4.52}, we have
 $$
\Psi_{\varepsilon}^{n+1}(t)\le c_{p,H,\alpha,T}\left(1+\left\|u_0\right\|_{L^p_{\lambda}(\Omega\times\mathbb{R})}^2+\left[\mathcal{N}_{\frac12-H,p}^*u_{0}\right]^2+\int_{0}^t(t-s)^{\frac{4H-3}{\alpha}}\Psi_{\varepsilon}^n(s)ds\right).
$$
Applying the generalized Gronwall's lemma (see, e.g., \cite[Lemma 15]{Dalang1999} or \cite[Lemma 1]{LHH2021}) yields
 $$
\sup_{n\ge1}\sup_{t\in[0,T]}\Psi_{\varepsilon}^{n}(t)\le c_{p,H,\alpha,T}<\infty.
$$
The desired result then  follows by Fatou's Lemma and an approximation argument,  replacing
$\liminf\limits_{\varepsilon\to0}$ with  $\liminf\limits_{n\to\infty}$  in the proof of Lemma \ref{lem 4.6}.
 The proof is complete.
 \end{proof}

\begin{lemma}\label{lem 4.7} Let $\frac{3-\alpha}{4}<H<\frac12$ and let $\lambda(x)$ be defined by \eqref{eq lambda}.
Let $u_{\varepsilon}$ be the approximate solution defined by \eqref{eq 4.39} and assume that $u_0(x)$ belongs to $\mathcal{Z}_{\lambda,0}^p$. Then  the following properties hold.
\begin{itemize}
\item[(i).]  If $p>\frac{2(\alpha+1)}{4H-3+\alpha}$, then
\begin{equation}\label{eq 4.55}
\left\|\sup_{t\in[0,T],\,x\in\mathbb{R}}\lambda^{\frac1{p}}(x)\mathcal{N}_{\frac12-H}u_{\varepsilon}(t,x)\right\|_{L^{p}(\Omega)}\le c_{\alpha,T,p,H}\left(\|u_{\varepsilon}\|_{\mathcal{Z}_{\lambda,T}^p}+1\right).
\end{equation}
\item[(ii).]  If $p>\frac{2(\alpha+1)}{2H+\alpha-2}$ and $0<\gamma<\frac{2H+\alpha-2}{2\alpha}-\frac{\alpha+1}{\alpha p},$ then
\begin{equation}\label{eq 4.56}
\left\|\sup_{t,t+h\in[0,T],x\in\mathbb{R}}\lambda^{\frac1{p}}(x)\left[u_{\varepsilon}(t+h,x)-u_{\varepsilon}(t,x)\right]\right\|_{L^{p}(\Omega)}\le c_{\alpha,T,p,H,\gamma}|h|^{\gamma}\left(\|u_{\varepsilon}\|_{\mathcal{Z}_{\lambda,T}^p}+1\right).
\end{equation}

\item[(iii).]  If $p>\frac{2(\alpha+1)}{2H+\alpha-2}$ and $0<\gamma<\frac{2H+\alpha-2}{2}-\frac{\alpha+1}{p},$ then
\begin{equation}\label{eq 4.57}
\left\|\sup_{t\in[0,T],\,x\in\mathbb{R}}\frac{u_{\varepsilon}(t,x)-u_{\varepsilon}(t,y)}{\lambda^{-\frac1p}(x)+\lambda^{-\frac1p}(y)}\right\|_{L^{p}(\Omega)}\le c_{\alpha,T,p,H,\gamma}|x-y|^{\gamma}\left(\|u_{\varepsilon}\|_{\mathcal{Z}_{\lambda,T}^p}+1\right).
\end{equation}
  \end{itemize}
 \end{lemma}
 
\begin{proof}
By Lemma \ref{lem 4.5}, we have $u_{\varepsilon}(t,x)\in \mathcal{Z}_{\lambda,T}^p$.  For any $\eta\in(0,1)$,   set
  \begin{equation*}
J^{\varepsilon}(\eta,r,x):=\int_0^{r}\int_{\mathbb{R}^2}(r-s)^{-\eta}G_{\alpha}(r-s,x-z)\sigma(u_{\varepsilon}(s,z))\rho_{\varepsilon}(z-y)dzW(ds,dy).
\end{equation*}
The stochastic Fubini's theorem   implies
\begin{align*}
u_{\varepsilon}(t,x)&=G_{\alpha}(t,\cdot)*u_0(x)+\frac{\sin(\pi \eta)}{\pi}\int_0^{t}\int_{\mathbb{R}}(t-r)^{\eta-1}G_{\alpha}(t-r,x-\xi)J^{\varepsilon}(\eta,r, x)dx dr\\
&=:u_1(t,x)+u_{2,\varepsilon}(t,x).
\end{align*}
Applying Proposition \ref{prop 4.2}(ii), (iii), (iv) to $u_{2,\varepsilon}(t,x)$ yields \eqref{eq 4.55}-\eqref{eq 4.57} without the constant term $1$. Replacing $u_{\varepsilon}(t,x)$ by $u_1(t,x)$ on the left-hand sides of \eqref{eq 4.55}-\eqref{eq 4.57}, we see that all these quantities   are  finite because  $u_0(x)\in \mathcal{Z}_{\lambda,0}^p$.  
The proof is complete.
\end{proof}

  \subsection{Proof of Theorems \ref{thm main1} and  \ref{thm main2}}
 
Now, we prove  Theorems \ref{thm main1} and  \ref{thm main2}, following the argument in Hu and Wang \cite[Section 4]{HW2022}.
\begin{proof}[Proof of Theorem \ref{thm main1}]
For simplicity, we still assume $\sigma(t,x,u)=\sigma(u)$.  From Lemma \ref{lem 4.5} and Lemma \ref{lem 4.7}(ii) and (iii),  it follows that the  probability measures induced by the processes $\{u_{\varepsilon},\varepsilon\in(0,1)\}$ on the space $\left(\mathcal{C}([0,T]\times\mathbb{R}), \mathfrak{B}(\mathcal{C}([0,T]\times\mathbb{R})), d_{\mathcal{C}}\right)$ are tight, by Theorem 4.4 in \cite{HW2022} (see also Section 2.4 in \cite{KS1991}). Hence,  there exists  a subsequence $\varepsilon_n\downarrow0$ such that $u_n:=u_{\varepsilon_n}$ converges weakly. By the Skorohod representation theorem,  there exists a probability space $(\widetilde{\Omega},\widetilde{\mathcal{F}}, \widetilde{\mathbb{P}})$ carrying the subsequence $\widetilde{u}_{n_j}$ and the noise $\widetilde{W}$ such that the finite-dimensional distributions of $(\widetilde{u}_{n_j},\widetilde{W})$ coincide with those of $(u_{n_j},W)$. Moreover, we have
\begin{equation}\label{eq 4.58}
\widetilde{u}_{n_j}(t,x)\to \widetilde{u}(t,x) \ \ \mbox{in}\ \ \left(\mathcal{C}([0,T]\times\mathbb{R}),d_{\mathcal{C}}\right)\  \  \widetilde{\mathbb{P}}\text{-almost surely,}
\end{equation}
for some  stochastic process $\widetilde{u}$ as $j\to\infty$. By Lemma \ref{lem 4.6},  we see that $\widetilde{u}$ belongs to $\widetilde{\mathcal{Z}}_{\lambda,T}^p$ with respect to the probability $\widetilde{\mathbb{P}}$. We will show that $\widetilde{u}$ is a weak solution to \eqref{SFHE}.

Let $\widetilde{\mathcal{F}}_t$  be the filtration generated by $\widetilde{W}$. Then $\widetilde{u}_{n_j}$ satisfies \eqref{SFHE} with $W$ replaced by $\widetilde{W}$ via Picard iteration:
  \begin{equation*}
      \widetilde{u}_{n_j}^{n+1}(t,x)=G_{\alpha}(t,\cdot)*u_0(x)+\int_0^t\int_{\mathbb{R}}[(G_{\alpha}(t-s,x-\cdot)\sigma(\widetilde{u}_{n_j}(s,\cdot)))*\rho_{\varepsilon_j}](y)\widetilde{W}(ds,dy).
   \end{equation*}
 Combining this with   \eqref{eq 4.58}, we obtain that  $\widetilde{u}$ is a mild solution to \eqref{SFHE} with $W$ replaced by $\widetilde{W}$.  Thus, we have proved the existence of a weak solution to \eqref{SFHE}.

   Moreover, for any $0<\gamma<\frac{2H+\alpha-2}{2}-\frac{\alpha+1}{p}$ and for any compact set $\textbf{T}\subset [0,T]\times \mathbb{R}$, Lemma \ref{lem 4.7}(ii) and (iii) implies that there exists  a constant $C$ such that
   $$
   \widetilde{\mathbb{E}}\left(\sup_{(t,x),(s,y)\in \textbf{T}}\left|\frac{\widetilde{u}(t,x)-\widetilde{u}(s,y)}{\left(\lambda^{-\frac1{p}}(x)+\lambda^{-\frac1{p}}(y)\right)\left(|t-s|^{\frac{\gamma}{\alpha}}+|x-y|^{\gamma}\right)}\right|^p\right)\le C\left(\|\widetilde{u}\|_{\mathcal{Z}_{\lambda,T}^p}+1\right)^p.
   $$
 This, together  with Kolmogorov's lemma,  implies the desired H\"older continuity.
  The proof is complete. 
\end{proof}

\subsection{Proof of Theorem \ref{thm main2}}
Since we have already proved the existence of a weak solution  to the nonlinear stochastic heat equation in  Theorem \ref{thm main1}, 
   pathwise uniqueness implies the existence  of a strong solution by the Yamada-Watanabe theorem, see \cite{IW1981} (in the SPDE setting, see, e.g., \cite{KS1991, Kurtz07}). Therefore,  it remains  to prove  the  pathwise uniqueness. We follow the same strategy as in  \cite{HHLNT2017, HW2022}, together with the crucial estimates in Proposition \ref{prop 4.2}. 
 
\begin{proof}[Proof of Theorem \ref{thm main2}]
Assume that $u$ and $v$ solve  \eqref{SFHE} and that  $u,v\in \mathcal{Z}_{\lambda,T}^p$.
Following \cite{HHLNT2017,HW2022},  we define the  stopping times as follows: for any $k\in \mathbb N$, 
\begin{align*}
T_k:=&\, \inf\Bigg\{t\in[0,T]:\sup_{0\le s\le t,x\in\mathbb{R}}\lambda^{\frac2p}(x)\mathcal{N}_{\frac12-H}u(s,x)\ge k, \\
&\, \,\,\,\,\,\,\,\,\,\, \,\,\,\,\,  \,\,\,\,\, \,\,\,\,\, \mbox{or } {\sup_{0\le s\le t,x\in\mathbb{R}}} \lambda^{\frac2p}(x)\mathcal{N}_{\frac12-H}v(s,x)\ge k\Bigg\}.
\end{align*}
Proposition \ref{prop 4.2}(ii) tells  us that $T_k\to T$ almost surely   as $k\to\infty$.  Denote
\begin{align*}
I_1(t):=&\, \sup_{x\in\mathbb{R}}\mathbb{E}\left[\textbf{1}_{t<T_k}|u(t,x)-v(t,x)|^2\right],\\
  I_2(t):=&\, \sup_{x\in\mathbb{R}}\mathbb{E}\left[\int_{\mathbb{R}}\textbf{1}_{t<T_k}|u(t,x)-v(t,x)-u(t,x+h)+v(t,x+h)|^2|h|^{2H-2}dh\right].
\end{align*}
 Using the same argument as in the proof of Theorem 1.6 in  \cite{HW2022}, we   obtain
\begin{align*}
I_1(t)\lesssim &\, k\int_0^t(t-s)^{\frac{2H-2}{\alpha}}\left[I_1(s)+I_2(s)\right]ds,\\
  I_2(t)\lesssim &\,  k\int_0^t(t-s)^{\frac{4H-3}{\alpha}}\left[I_1(s)+I_2(s)\right]ds.
\end{align*}
Consequently,
$$
I_1(t)+I_2(t)\lesssim k\int_0^t(t-s)^{\frac{4H-3}{\alpha}}\left[I_1(s)+I_2(s)\right]ds.
$$
Then Gronwall's inequality implies $I_1(t)+I_2(t)=0$ for $t\in[0,T]$. In particular,  
$$
\mathbb{E}\left[\textbf{1}_{t<T_k}|u(t,x)-v(t,x)|^2\right]=0.
$$
Thus $u(t,x)=v(t,x)$ almost surely on $\{t<T_k\}$ for all $k\ge1$.  Letting $k\to\infty$ shows that  $u(t,x)=v(t,x)$ almost surely for every $t\in[0,T]$ and $x\in\mathbb{R}$.

The existence of a H\"older continuous modification of the solution follows from Theorem \ref{thm main1}.  
   The proof is complete.
 \end{proof}
 
 \vskip0.2cm

\noindent{\bf Acknowledgments}  The research of R. Wang is partially supported by the NSF of Hubei Province (No. 2024AFB683).

\vskip0.8cm

    \end{document}